\documentclass[11pt,reqno]{amsproc}
\usepackage[margin=1in]{geometry}
\usepackage{amsmath, amsthm, amssymb,enumitem}
\usepackage[colorlinks=true, pdfstartview=FitV, linkcolor=blue,citecolor=blue, urlcolor=blue]{hyperref}
\usepackage[abbrev,lite,nobysame]{amsrefs}
\usepackage{times}
\usepackage[usenames,dvipsnames]{color}
\usepackage{mathtools}
\usepackage{dsfont}

\def\eps{\varepsilon}
\def\e{{\rm e}}
\def\Re{{\rm Re}}
\def\dg{{D_y}}
\def\dc{{d}}
\def\a{{a}}
\def\b{{b}}
\def\tPhi{\Phi^{(1)}}
\def\tPsi{\Psi^{(1)}}
\def\ttPhi{\Phi^{(2)}}
\def\ttPsi{\Psi^{(2)}}
\def\ttpsi{\psi^{(2)}}

\def\dist{{\rm dist}}
\def\sign{{\rm sign}}
\def\spt{{\rm spt}}

\def\dd{{\rm d}}
\def\ddt{{\frac{\dd}{\dd t}}}

\def\R {\mathbb{R}}

\def\J {{\mathcal J}}

\def\ZZ {{\mathbb Z}}

\def \l {\langle}
\def \r {\rangle}
\def\T {{\mathbb T}}

\def\de{{\partial}}

\def\p {{\partial}}
\def\dt {{\partial_{t}}}

\newtheorem{proposition}{Proposition}[section]
\newtheorem{theorem}[proposition]{Theorem}
\newtheorem{corollary}[proposition]{Corollary}
\newtheorem{lemma}[proposition]{Lemma}
\theoremstyle{definition}

\newtheorem{remark}[proposition]{Remark}

\numberwithin{equation}{section}

\title[Degenerate shear flows and circular flows]{On degenerate circular and shear flows: the point vortex and power law circular flows}

\author[M. Coti Zelati and C. Zillinger]{Michele Coti Zelati and Christian Zillinger}

\address{Department of Mathematics, Imperial College London, London, SW7 2AZ, UK}
\email{m.coti-zelati@imperial.ac.uk}

\address{Department of Mathematics, University of Southern California, Los Angeles, California 90089, USA}
\email{zillinge@usc.edu}



\begin{document}

\begin{abstract}
We  consider  the  problem  of  asymptotic  stability  and  linear  inviscid  damping  for  perturbations  of  a  point  vortex  and
similar degenerate circular flows.  Here, key challenges include the lack of strict monotonicity and the necessity of working
in weighted Sobolev spaces whose weights degenerate as the radius tends to zero
or infinity.
Prototypical examples are given by circular flows with power law singularities or zeros
as $r\downarrow 0$ or $r \uparrow \infty$.
\end{abstract}


\maketitle

\tableofcontents

\section{Introduction}

When considering the 2D Euler equations close to Couette flow on $\T \times \R$,
it is a classical result by Orr that the linearized problem for the vorticity
reduces to free transport. Using the explicit solution, it can then be shown
that the vorticity weakly, but not strongly, converges to its $x$ average and
that the associated (perturbation to the) velocity field strongly converges to a
shear flow with sharp algebraic decay rates provided the initial perturbation is
sufficiently regular. For this special case, a precise description of this is
given by the explicit Fourier characterization of the velocity field, namely
\begin{align*}
  \tilde{v}(t,k,\eta+kt)= (-i(\eta+kt), ik) \frac{1}{k^{2}+(\eta-kt)^{2}}\tilde{\omega}(0,k,\eta),
\end{align*}
where $ t,\eta \in \R, k \in \ZZ$. By Plancherel's theorem and noting that
$(k,\eta)\mapsto (k,\eta+kt)$ is an isometry, we observe the decay and
convergence of $\tilde{v}$ as well as its precise dependence on the initial
data.

If the underlying profile is not of this special form, the problem is not
explicitly solvable anymore and much effort has to be invested to establish
similar results. For a discussion of the literature we refer to
\cite{bedrossian2015dynamics} and just briefly mention the following works:
\begin{itemize}
\item In \cite{bedrossian2015inviscid}, nonlinear inviscid damping for Gevrey regular perturbations around Couette
  flow on $\T \times \R$ is established using methods of pseudo-differential calculus,
  paraproducts and scattering methods. In particular, their analysis yields a
  fine description of the nonlinear dynamics and ``echos'' and has been extended
  to many further settings such as the 3D or viscous setting with coauthors.
\item In \cite{Zill3}, using similar methods, the second author established
  linear inviscid damping for, roughly speaking, Bilipschitz flow profiles also
  in the setting of domains with boundary such as a finite periodic channel $\T
  \times [0,1]$ with impermeable walls. Here, boundary effects impose strong
  limits on the achievable regularity of the linearized problem to fractional
  Sobolev spaces and as consequence also on the nonlinear problem. In
  \cite{Zill5}, these results are extended to weighted Sobolev spaces, following
  the works of \cite{Zhang2015inviscid} and the setting of circular flows. Here,
  in order to avoid degeneracies results are limited to annuli
  $B_{R_{2}}\setminus B_{R_{1}}$ with $0<R_{1}<R_{2}<\infty$.
\item In \cite{Zhang2015inviscid}, a very different spectral approach is used
  instead to establish linear inviscid damping under different/weaker
  conditions. Recently, these methods have further been able to establish
  similar results for flows close to Poiseuille flow \cite{WZZvorticitydepl} and for 
  the Kolmogorov flow \cite{WZZkolmogorov}.
\end{itemize}

The main motivating example of this work is given by perturbations of the
velocity profile $\frac{1}{r}e_{\theta}$ on $\R^{2}\setminus \{0\}$ of a point
vortex at the origin.
While for the unperturbed problem, solutions are explicit, for perturbations
the degenerate behavior as $r$ tends to $0$ or $\infty$ makes the question of
(asymptotic) stability and damping mathematically very challenging.

As the main results of this article, we establish linear inviscid damping,
stability and scattering for a class of \emph{mildly degenerate} flows (c.f.
Section \ref{sec:mildly_degen}).
This class includes circular flows with power law $U(r)\sim r^{\alpha}, \alpha \in \R$ as
$r\downarrow 0$ and $U(r)\sim r^{\beta}, \beta\in \R$ as $r\uparrow \infty$ as
well as degenerate shear flows like $U(y)=\e^{y}$ in a periodic channel $\T
\times I$.
As the main stability assumption, here we do not require strict monotonicity, but rather require that
\begin{align}
   \sign(U') U''' \leq c |U'|,
\end{align}
which is natural in view of explicit and scaling results for polynomial and
exponential shear flow profiles on bounded intervals.
Due to technical obstructions and the locally perturbative nature of our method,
some further conditions are required on the underlying flow (c.f. Section \ref{sec:mildly_degen})
We remark that this condition is different from Arnold's stability criterion, which
is given by a control of $U''/U$ from above and below.

Stability of radially symmetric, strictly monotone decreasing distributions of 
vorticity was studied recently  in \cite{BCZVvortex2017}, using a spectral approach,
devising interesting phenomena such as \emph{vortex axysimmetrization}  \emph{vorticity depletion} (see \cite{Euler_stability} and \cite{WZZvorticitydepl}). 
Our work can be seen as complementary to this one: we do not require monotonicity
of the profile and we can handle vorticity profiles that blow up at $r=0$.

In the following we recall the linearized Euler equations for shear flows and
circular flows and introduce our notion of \emph{mildly degenerate} flows.
Our main results are then stated in Subsection \ref{sec:main-results}.

\subsection{The linearized Euler equations}
\label{sec:linEuler}
We consider the linearization of the 2D incompressible Euler equations
\begin{align*}
  \partial_{t} \omega + v \cdot \nabla \omega =0
\end{align*}
around
\begin{enumerate}
\item A shear flow $v=(U(y), 0), \omega=-U'(y)$ on $\T \times I$, where $I$ is a
  possible infinite interval.
\item A circular flow $v=u(r) r e_\theta, \omega =\frac{1}{r}\partial_r (r^{2}u(r))$ on $\R^2\setminus
  \{0\}$ or $\R^2\setminus B_{r_1}(0)$ or $B_{r_2}(0)\setminus B_{r_1}(0)$.
\end{enumerate}

The linearized Euler equations around the shear flow are given by
\begin{align}
  \label{eq:shear}
  \begin{split}
  \partial_{t} \omega + U(y)\partial_x \omega &= U''(y)\partial_x \psi, \\
  \Delta \psi &=\omega,
  \end{split}
\end{align}
where for an infinite interval we prescribe $\nabla \psi \in L^2$ and on the
boundary require that $\partial_x \psi=0$ (impermeable walls).

For the circular flow, we first consider the linearized problem in polar coordinates
\begin{align*}
  \dt \omega + u(r)\p_{\theta} \omega &= b(r) \p_{\theta} \omega, \\
  (\p_{r}^{2}+ \frac{1}{r}\p_{r} + \frac{1}{r^{2}}\p_{\theta}^{2})\phi&=f,
\end{align*}
where $b(r)= -\frac{1}{r}\p_{r} (\frac{1}{r}\p_{r}(r^{2} u(r)))$.
Here, the distinguished cases with $b(r)=0$ for $r>0$ are given by the Taylor-Couette flow $u(r)=C_{1} +
\frac{C_{2}}{r^{2}}$ corresponding to constant angular velocity for $C_{2}=0$ and a point
vortex for $C_{1}=0$, respectively.

Introducing log-polar coordinates
$(x,y)=\e^s(\cos\theta, \sin\theta)$ as well as relabeling
$\omega(t,s,\theta):=\e^{2s}\omega(t,s,\theta)$, $B(e^{s})=e^{-2s}b(e^{s})$ we obtain
\begin{align}
  \label{eq:circular}
  \begin{split}
  \dt \omega + u(\e^s)\p_\theta \omega  &= B(\e^s)\p_\theta \psi , \\
  (\p_{\theta}^2 + \p_{s}^2) \psi &=\omega, 
  \end{split}
\end{align}
where $B(\e^{s})=\p_{s} (\e^{-2s}\p_{s}(\e^{2s}u(\e^{s})))$.

We can reformulate both problems discussed above in a unified fashion as follows. Let $I$ be a (possibly infinite)
interval, and let $U,B:I\to \R$ be given functions. We study the behavior of solutions
$\omega: \T\times I\to \R$
to the linear problem
\begin{align}\label{eq:Euler}
\begin{cases}
\de_t \omega + U(y)\de_x \omega = B(y) \de_x \psi, \quad &\mbox{in } \T \times I, \\
-\Delta \psi = \omega, &\mbox{in } \T \times I,\\
\de_x \psi= 0, &\mbox{on } \T \times \de I,
\end{cases}
\end{align}
with initial datum
\begin{align}
\omega(0,x,y)=\omega^{in}(x,y),  \quad \mbox{in } \T \times I.
\end{align}

\subsection{Main result}
\label{sec:main-results}

Our main result consists of establishing inviscid damping, stability and scattering
for flows satisfying suitable conditions, which we call \emph{mildly degenerate flows}
(c.f. Section \ref{sec:mildly_degen}). Prototypical examples here are given by
shear flows of the form $U(y)=\e^{\alpha y}$ and circular flows with power law
behavior as $r \downarrow 0$ and $r \uparrow \infty$.

\begin{theorem}\label{thm:main}
  Suppose that $U(y), B(y)$ are mildly degenerate flows on $\T_{L} \times I$.
  Then there exists a symmetric positive operator $A(t):L^2 \rightarrow L^2$ with $C^1$
  dependence on $t$ such that any solution $\omega$ of \eqref{eq:Euler} satisfies
  \begin{align}
    \langle \omega(t), A(t) \omega(t) \rangle_{L^2} &\approx \|\omega(t)\|_{L^2}^2, \\
    \ddt \langle \omega(t), A(t) \omega(t) \rangle_{L^2} &\leq -c \int |U'(y)| |\nabla \psi (t)|^{2} \dd x \dd y \leq 0.
  \end{align}
  In particular, this implies $L^2$ stability and
  \begin{align}
    \int |U'(y)| |\nabla \psi (t)|^{2} \dd x \dd y  \in L^1_t.
  \end{align}
  If further, the initial data is in $H^{1}$ or $H^{2}$, then also 
  \begin{align*}
    \|\omega(t, x- t U(y), y)\|_{H^{1}} \leq C \|\omega^{in}\|_{H^{1}}, \\
    \|\min(1, \dist(y, \p I)) \de_{yy}\omega(t, x- t U(y), y)\|_{L^{2}} \leq C \|\omega^{in}\|_{H^{2}},
  \end{align*}
  and
  \begin{align*}
    \int |U'(y)| |\nabla \psi (t)|^{2} \dd x \dd y \leq \frac{C}{\l t\r^{2}} \|\omega^{in}\|^2_{H^{1}}, \\
    \int |U'(y)| |\psi (t)|^{2} \dd x \dd y \leq \frac{C}{\l t\r^4}\|\omega^{in}\|^2_{H^{2}}.
  \end{align*}
\end{theorem}

These results extend the inviscid damping and scattering results of \cite{Zill6}
to the \emph{mildly degenerate} setting, which in particular allows for $U'$ to
converge to zero or infinity as $y$ approaches the boundary.

\subsection{Mildly degenerate flows}
\label{sec:mildly_degen}
We call the coefficient functions $U,B$ \emph{mildly degenerate} with
  respect to a torus $\T_{L}$ if the following conditions hold.

\begin{enumerate}[label={(H\arabic*)}]
\item\label{H1} There exists a constant $\gamma_0\in(0,1)$ such  that 
\begin{align}\label{eq:mildy1}
k^{2}|U'(y)| - \frac12\sign(U'(y))U'''(y) \geq \gamma_0 k^{2}|U'(y)|, \qquad \forall y\in I, k \in \frac{2\pi}{L}\mathbb{Z}\setminus\{0\}.
\end{align}
\item \label{H2}There exists $\eps_0 < \infty$  such that
\begin{align}\label{eq:mildy2}
|B(y)|+ |B'(y)| +|U''(y)|\leq \eps_0 |U'(y)|, \qquad \forall y\in I.
\end{align}

\item\label{H3} 
There exists a collection of smooth functions $\chi_j$ such that
\begin{align}
\sum_{j\in\J}\chi_j^2=1, \qquad \forall y\in I.
\end{align}
and so that the support of each function $\chi_{j}$, $\spt \chi_j =: I_j$ is an
interval with
\begin{align}\label{eq:mildly31}
\inf_{j\in \J}|I_j|=:\kappa >0,
\end{align}
and
\begin{align}\label{eq:mildly32}
\sup_{j\in\J}\left\|\frac{\max_{y\in I_j} U'(y)}{U'(y)}\right\|_{L^\infty(I_j)}<\infty.
\end{align} 
In light of \eqref{eq:mildly31}, this further implies that for some
  constant $C$ 
\begin{align}\label{eq:part1}
  \sup_{j\in\J} \|\chi_j\|_{W^{1,\infty}}< C (1+ \frac{1}{\kappa}), \qquad   \sup_{j\in\J} \|\chi_j\|_{W^{2,\infty}}< C^{2} (1+ \frac{1}{\kappa})^{2}.
\end{align}

\end{enumerate}
Unless otherwise stated, we will assume throughout the article that $U$ and $B$
are mildly degenerate.
Furthermore, for some estimates we will require a perturbation condition:
\begin{enumerate}[resume, label={(H\arabic*)}]
\item \label{P} Let $U,B$ be mildly degenerate and let $\epsilon_{0}, \kappa$
  be as in the definition. Then we say that the flow is perturbative or satisfies a smallness
  condition if $ \epsilon_{0} \leq \frac{k}{2}$ and $ \epsilon_{0} C(1+1/\kappa)^{2} \leq
  \frac{\gamma_{0} k^{2}}{4}$ (c.f. Section \ref{sec:auxil-funct}).
\end{enumerate}
Let us briefly comment on these conditions:
\begin{itemize}
\item The condition \ref{H1} yields that the Laplacian is a strictly elliptic
  operator on $L^{2}(\T_{L} \times I, |U'(y)| \dd x \dd y)$ for any interval $I$ (c.f.
  Lemma \ref{lem:weighted_elliptic}).
  Combined with the estimates in \ref{H2}, this allows us to control the
  right-hand-side of equation \eqref{eq:Euler} in a weighted negative Sobolev
  space.
\item A prototypical example of a mildly degenerate flow is given by
  $U(y)=\e^{\alpha y}, B(y)=C \e^{\alpha y}$, where condition \ref{H1} imposes a
  constraint on $\alpha$. Considering that $\sinh(ky) \sin(kx)$ is in the kernel
  of the Laplacian this condition seems necessary.
  Condition \ref{H3} then corresponds to a partition of $I$ using (dyadic) level
  sets of $U'$. Using a newly introduced localization procedure, we establish
  our main results by constructing localized pseudodifferential weights adapted
  to this covering.
\item As we discuss in following, this prototypical setting also allows us to
  consider circular flows with power law singularities or zeros as $r\downarrow
  0$ or $r \uparrow \infty$ on $\R^{2}\setminus \{0\}$.
\item If for instance $U(y)=\cos(y)$ or $U(y)=y^{2}$, the condition
  \eqref{eq:mildly31} can relaxed to allow shrinking dyadic intervals. However,
  in that case further cancellation or decay due to symmetry, Hardy's inequality
  or higher decay of $B(y)$ has to be assumed. The remaining obstacle to treat
  such non-mildly degenerate flows is then in improving \ref{H2} to
  \begin{align}
    |B'(y)| \leq \eps_{0} |U'(y)|,
  \end{align}
  which would require improved commutator estimates. 
\item The smallness condition \ref{P} quantifies closeness to Taylor-Couette
  flow and is weaker the larger $k$ is. It is imposed so that the
  right-hand-side in \eqref{eq:Euler} in the end yields a small perturbation to
  the transport semigroup. It is not optimal and, indeed, as shown in
  \cite{Zhang2015inviscid} if instead of estimating by absolute values one exploits signs and
  cancellations, a weaker condition can be obtained.

  The condition on $\kappa$ in our applications follows by \ref{H1}, since we
  choose the sets $I_{j}$ according to (dyadic) level sets of $U'$ and can
  control the growth and decay of $U'$.
\end{itemize}

By expanding the solution $\omega$ to \eqref{eq:Euler} as a Fourier series in the $x$ variable, namely
\begin{align}
\omega(t,x,y)=\sum_{k\neq 0} \omega_k(t,y)\e^{ikx}, \qquad  \psi(t,x,y)=\sum_{k\neq 0} \psi_k(t,y)\e^{ikx},
\end{align} 
we can perform a $k$-by-$k$ analysis of the linearized equations. Thus, studying \eqref{eq:Euler} is equivalent
to analyzing the collection of one-dimensional problems
\begin{align}\label{eq:Eulerk}
\begin{cases}
\de_t \omega +ik U(y) \omega = ik B(y)  \psi, \quad &\mbox{in } I, \\
-\Delta_k \psi:=-(-k^2+\de_{yy})\psi = \omega, &\mbox{in } I,\\
ik\psi= 0, &\mbox{on } \de I,
\end{cases}
\end{align}
and
\begin{align}
\omega(0,y)=\omega^{in}(y),  \quad \mbox{in } I,
\end{align}
where we do not keep track of the index $k$ to simplify notation. It is clear that the $k=0$ mode is conserved by the 
above equation, and that the analysis is the same for positive and negative $k$. Thus, in what follows we restrict ourselves
to the case $k\geq 1$. An equivalent formulation of the above problem is sometimes called \emph{scattering formulation},
and can be derived as follows. We denote by $S(t)$
the solution operator corresponding to the transport operator $U\de_x$. 
More explicitly, 
\begin{align}
g(x,y)\mapsto (S(t)g)(x,y)=g(x+tU(y),y).
\end{align}
Let $U,B$ be given, and let $\omega,\psi$ be the solution to \eqref{eq:Eulerk}. Accordingly, 
we define the \emph{scattered vorticity} and \emph{scattered streamfunction}
by
\begin{align}
F(t)=S(-t)\omega(t), \qquad \Psi(t)=S(-t)\psi(t).
\end{align}
It is not hard to check that $\omega,\psi$ solve \eqref{eq:Eulerk} if and only if $F,\Psi$
solve
\begin{align}\label{eq:Eulerkscat}
\begin{cases}
\de_t F =ikB(y)\Psi,   \quad &\mbox{in } I, \\
E_t \Psi:=(-k^2+(\de_y-iktU'(y))^2)\Psi  = F, &\mbox{in } I,\\
ik\Psi= 0, &\mbox{on } \de I,
\end{cases}
\end{align}
with
\begin{align}
F(0,y)=\omega^{in}(y),  \quad \mbox{in } I.
\end{align}
The function $F$ is sometimes called ``profile'', and it is often studied in dispersive equations. In this context,
it is the object which measures the difference between the passive scalar and full linearized (or nonlinear) dynamics.

\begin{remark}
\label{remark:boundary_data}  
We note that the equations decouple in $k$ and the evolution of $F$ is trivial
for $k=0$. Hence, we consider $k \in \frac{2\pi}{L} \mathbb{Z} \setminus\{0\}$ as a
given parameter.
In particular, in this case $\Psi|_{\de I}=0$ and thus
\begin{align}
  \de_{t} F|_{\de I}=0
\end{align}
for all times and boundary values are preserved.
That is, for all $t>0$
\begin{align}
  F(t)|_{\p I}= \omega^{in}|_{\p I}.
\end{align}
We further remark that, since the spectral gap of $E_{t}$ involves
$k^{2}$, $ik \Psi$ asymptotically scales as $|k|^{-1}$. Thus, and in view of
condition \ref{H1}, estimates for larger $k$ are simpler to establish than for
smaller $k$.
\end{remark}

Our main result \eqref{thm:main} has a natural formulation frequency by frequency, which we state here below.

\begin{theorem}\label{thm:L2}
There exists a symmetric positive operator $A(t):L^2 \rightarrow L^2$, with $C^1$
  dependence on $t$, such that any solution $\omega$ of \eqref{eq:Eulerk} satisfies
  \begin{align}
    \frac1C\|\omega(t)\|_{L^2}^2 \leq\l \omega(t), A(t) \omega(t) \r_{L^2} \leq C \|\omega(t)\|_{L^2}^2, \qquad \forall t\geq0,
  \end{align}
  for some constant $C\geq 1$, and
    \begin{align}
    \ddt \l \omega(t), A(t) \omega(t) \r_{L^2} +\frac12 \int_{I} |U'(y)| |\nabla_k \psi (t,y)|^{2}  \dd y \leq 0, \qquad \forall t\geq0.
  \end{align}
  In particular, this implies $L^2$ stability and
  \begin{align}
    \int_{I} |U'(y)| |\nabla_k \psi (y)|^{2} \dd y  \in L^1_t(0,\infty).
  \end{align}
    If further, the initial data is in $H^{1}$ or $H^{2}$, then also 
  \begin{align*}
    \|\e^{ikUt}\omega(t, y)\|_{H^{1}} \leq C \|\omega^{in}\|_{H^{1}}, \\
    \|\min(1, \dist(y, \p I)) \de_{yy}\e^{ikUt}\omega(t, y)\|_{L^{2}} \leq C \|\omega^{in}\|_{H^{2}},
  \end{align*}
  and
  \begin{align*}
    \int |U'(y)| |\nabla_k \psi (t)|^{2}  \dd y \leq \frac{C}{\l kt\r^{2}} \|\omega^{in}\|^2_{H^{1}}, \\
    \int |U'(y)| |\psi (t)|^{2} \dd y \leq \frac{C}{\l kt\r^4}\|\omega^{in}\|^2_{H^{2}}.
  \end{align*}
\end{theorem}

Higher order stability can be stated in terms of modified differential operators. Let 
\begin{align}\label{eq:newdiff}
\dc(y):=\sum_{j\in\J}\frac{\|U'\|_{L^\infty(I_j)}}{U'(y)}\chi_j(y), \qquad \dg:=\dc(y)\de_y. 
\end{align}
Our choice of $\dg$ here is determined by satisfying two opposing objectives.
On the one hand, we want an operator that is very similar to the usual derivative $\de_y$.
On the other hand, $\dg$ should have good commutation properties with
$\de_{y}-iktU'$, which leads to considering $(U')^{-1}\de_y$.

\subsection{Notation and conventions}
\begin{itemize}
\item $C\geq 1$ is a generic positive constant independent of $k$.
\item $\l\cdot,\cdot\r$ is the $L^2$ scalar product in $y\in I$.
\item $\nabla_k=(ik,\de_y)$
\item $[A,B]=AB-BA$
\item $\l x\r=\sqrt{1+x^2}$
\end{itemize}

\section{Localization and \texorpdfstring{$L^2$}{L2} stability}\label{sec:simple_case}
Let $\omega$ be a solution to \eqref{eq:Eulerk}.
Then we claim that $\omega$ satisfies
\begin{align}
\ddt \|\omega\|^2_{L^2} \leq 2 \int_I |B'(y)||\nabla_k\psi(y)|\dd y\leq C \int_I |U'(y)||\nabla_k\psi(y)|\dd y.
\end{align}
Indeed, using the anti-symmetry of $ikU$ and $ikb$ and integrating by parts, we have
\begin{align}
\ddt \|\omega\|^2_{L^2} &=2\Re\l ik B\psi, \omega\r=-2\Re\l ik B\psi, \Delta_k \psi\r
=-2\Re\l ik B\psi, \de_{yy} \psi\r=2\Re\l ik B'\psi, \de_y \psi\r,
\end{align}
so the claim follows by the Cauchy--Schwarz inequality and \eqref{eq:mildy2}. 
Notice that here we used that $\psi=0$  on $\de I$ in the integration by parts.

In the following, we need a more localized version of the above estimate,
adapted to the partition $\chi_j$, where we morally would want to replace $I$ by
$I_{j}=\text{supp}(\chi_{j})$.
As the Biot Savart law is \emph{non-local}, such an estimate can not hold
exactly. However, we show that in summed sense and in suitably modified and
weighted spaces such a localization can indeed be established.

We begin with a preliminary key lemma establishing weighted elliptic estimates.

\begin{lemma}\label{lem:weighted_elliptic}
Let $J\subset I$ be an interval. For any complex-valued $g\in H^1_0(J)$ and any $k\geq 1$ there holds
\begin{align}
- \l  |U'| g, \Delta_k g\r \geq \gamma_{0}\int_J |U'(y)||\nabla_k g(y)|^2\dd y.
\end{align}
In particular,  the left-hand-side is a positive definite bilinear form in $\nabla_k g$.
\end{lemma}

\begin{proof}
We integrate by parts twice and use that $g$ vanishes at the boundary to obtain
\begin{align}
- \l g |U'|, \Delta_k g\r  
&= \int_J \left[|U'(y)| |\nabla_k g(y)|^{2} +\frac12 \de_{y}|U'(y)| \de_{y}|g(y)|^{2} \right]\dd y \notag\\
&= \int_J\left[ |U'(y)| |\nabla_k g(y)|^{2}- \frac{1}{2}|g(y)|^{2} \de_{yy}|U'(y)|\right]\dd y.
\end{align}
Since $U$ and $B$ are mildly degenerate, in the sense of weak derivatives we have
\begin{align}\label{eq:weakder}
\de_{yy}|U'|=\sign(U')U''' +2 \sum_{\bar{y}: U'(\bar{y})=0} \delta_{\bar{y}} U''(\bar{y})=\sign(U')U''',
\end{align}
and therefore we can exploit \eqref{eq:mildy1} to deduce that 
\begin{align}
- \l g |U'|, \Delta_k g\r   
 &= \int_J\left[ |U'(y)| |\nabla_k g(y)|^{2}- \frac{1}{2}|g(y)|^{2} \sign(U'(y))U'''(y)\right]\dd y\notag\\
&= \int_J\left[ \left(k^2|U'(y)|- \frac{1}{2}\sign(U'(y))U'''(y) \right) |g(y)|^{2}+|U'(y)| |\de_y g(y)|^{2}\right]\dd y\notag\\
&\geq \gamma_{0}\int_J |U'(y)||\nabla_k g(y)|^2\dd y,
\label{eq:elli1}
\end{align}
which is what we need to conclude the proof.
\end{proof}

\subsection{Localized potentials}\label{sub:localizedpote}
In order to properly localize the streamfunction in \eqref{eq:Eulerk}, we introduce the following auxiliary problems.
For any $j\in \J$, let 
\begin{align}\label{eq:Ajs}
A_j(t): L^2(I_j) \rightarrow L^2(I_j)
\end{align} be a symmetric operator such that $t\mapsto A_j(t)$ is 
continuously differentiable. Define the localized potentials $\psi_j, \psi_{j,A}$ by 
\begin{equation}\label{eq:local1}
\begin{cases}
-\Delta_k \psi_j = \chi_j \omega, \quad &\mbox{in } I_j , \\
\psi_{j}= 0, &\mbox{on } \de I_j, 
\end{cases}
\end{equation}
and
\begin{equation}\label{eq:local2}
\begin{cases}
-\Delta_k \psi_{j,A} = A_j(t)\chi_j \omega, \quad &\mbox{in }  I_j , \\
\psi_{j,A}= 0, &\mbox{on }  \de I_j,
\end{cases}
\end{equation}
respectively. When referring to the scattering formulation \eqref{eq:Eulerkscat}, we will make use of the 
scattered localized potentials $\Psi_j, \Psi_{j,A}$ defined by 
\begin{equation}\label{eq:local1scat}
\begin{cases}
-E_{t} \Psi_j = \chi_j F, \quad &\mbox{in } I_j , \\
\Psi_{j}= 0, &\mbox{on } \de I_j, 
\end{cases}
\end{equation}
and
\begin{equation}\label{eq:local2scat}
\begin{cases}
-E_{t} \Psi_{j,A} = A_j(t)\chi_j F, \quad &\mbox{in }  I_j , \\
\Psi_{j,A}= 0, &\mbox{on }  \de I_j.
\end{cases}
\end{equation}
The following lemma  relates a term that commonly arises in our computations with its localized version.
\begin{lemma}\label{lem:local_elliptic}
For every $k\geq 1$ we have
\begin{align}
\int_I |U'(y)||\nabla_k \psi(y)|^2\dd y \leq C\sum_{j\in \J}\int_{I_j} |U'(y)||\nabla_k \psi_j(y)|^2\dd y,
\end{align}
for some $C\geq 1$, independent of $k$.
\end{lemma}

\begin{proof}
From Lemma \ref{lem:weighted_elliptic}, we have
\begin{align}\label{eq:intpart}
\gamma_{0} \int_I |U'(y)||\nabla_k \psi(y)|^2\dd y\leq -\l \Delta \psi, |U'|\psi\r
=-\sum_{j\in\J}\l \Delta \psi_j, \chi_j|U'|\psi\r.
\end{align}
Recalling \eqref{eq:local1} and using the product formula
\begin{align}
\Delta_k (|U'|\psi\chi_j)=\de_{yy}|U'|\psi\chi_j+2\de_y|U'|\de_y(\psi\chi_j)+|U'|[\Delta_k\psi_j+2\de_y\psi\chi_j'+\psi\chi_j''],
\end{align}
for each fixed $j\in\J$, several integration by parts yield
\begin{align}\label{eq:estim0}
\l \Delta_k \psi_j, |U'|\psi\chi_j\r
&=\l  \psi_j, \de_{yy}|U'|\psi\chi_j+2\de_y|U'|\de_y(\psi\chi_j)+|U'|[\Delta_k\psi_j+2\de_y\psi\chi_j'+\psi\chi_j'']\r\notag\\
&=\l  \Delta_k\psi_j,|U'|\psi_j\r+\l  \psi_j, \de_{yy}|U'|\psi\chi_j\r
+2\l  \de_y|U'|,\psi_j\de_y(\psi\chi_j)\r
+\l  |U'|\psi_j,2\de_y\psi\chi_j'+\psi\chi_j''\r\notag\\
&=\l  \Delta_k\psi_j,|U'|\psi_j\r+\l  \psi_j, \de_{yy}|U'|\psi\chi_j\r
-2\l  |U'|,\de_y(\psi_j\de_y(\psi\chi_j))\r
+\l  |U'|\psi_j,2\de_y\psi\chi_j'+\psi\chi_j''\r\notag\\
&=\l  \Delta_k\psi_j,|U'|\psi_j\r-2\l \de_{yy}\psi \chi_j,|U'|\psi_j\r
+\l  \psi_j, \de_{yy}|U'|\psi\chi_j\r\notag\\
&\quad- 2\l\de_y \psi_j ,|U'|\de_y\psi \chi_j\r- 2\l  \de_y\psi_j,|U'|\psi\chi'_j\r- 2\l\psi_j  ,|U'|\de_y\psi \chi'_j\r-\l \psi_j,|U'|\psi\chi''_j\r\notag\\
&=-\l  \Delta_k\psi_j,|U'|\psi_j\r+2k^2\l\psi \chi_j,|U'|\psi_j\r
+\l  \psi_j, \de_{yy}|U'|\psi\chi_j\r\notag\\
&\quad- 2\l\de_y \psi_j ,|U'|\de_y\psi \chi_j\r- 2\l  \de_y\psi_j,|U'|\psi\chi'_j\r- 2\l\psi_j  ,|U'|\de_y\psi \chi'_j\r-\l \psi_j,|U'|\psi\chi''_j\r\notag\\
&=-\l  \Delta_k\psi_j,|U'|\psi_j\r+\l  \psi_j, \de_{yy}|U'|\psi\chi_j\r\notag\\
&\quad- 2\l\nabla_k \psi_j ,|U'|\nabla_k\psi \chi_j\r- 2\l  \de_y\psi_j,|U'|\psi\chi'_j\r- 2\l\psi_j  ,|U'|\de_y\psi \chi'_j\r-\l \psi_j,|U'|\psi\chi''_j\r
\end{align}
We estimate the right-hand side above term by term. Appealing to \eqref{eq:part1}, it is not hard to see that 
\begin{align}
&|\l\nabla_k \psi_j ,|U'|\nabla_k\psi \chi_j\r| +|\l  \de_y\psi_j,|U'|\psi\chi'_j\r|+ |\l\psi_j  ,|U'|\de_y\psi \chi'_j\r|+|\l \psi_j,|U'|\psi\chi''_j\r|\notag\\
&\qquad\leq C\int_{I_j} |U'(y)| \left[|\nabla_k \psi_j(y)|(|\nabla_k \psi(y)|+|\psi(y)|) + | \psi_j(y)|(|\nabla_k \psi(y)|+|\psi(y)|) \right]\dd y\notag\\
&\qquad\leq C\int_{I_j} |U'(y)| |\nabla_k \psi_j(y)||\nabla_k \psi(y)|\dd y.
\end{align}
Moreover, in view of \eqref{eq:mildy1}, \eqref{eq:weakder} and the mild degeneracy assumptions, we have that
\begin{align}
|\l  \psi_j, \de_{yy}|U'|\psi\chi_j\r|=|\l  \psi_j, \sign(U')U'''\psi\chi_j\r|\leq C\int_{I_j}|U'(y)||\nabla_k \psi_j(y)||\nabla_k \psi(y)|\dd y.
\end{align}
Thus,  \eqref{eq:intpart}, \eqref{eq:estim0}  and Lemma \ref{lem:weighted_elliptic} yield
\begin{align}
 &\gamma_{0}\int_I |U'(y)||\nabla_k \psi(y)|^2\dd y\\
 &\qquad\qquad\leq \sum_{j\in \J}\l  \Delta_k\psi_j,|U'|\psi_j\r+ C\sum_{j\in \J}\int_{I_j}|U'(y)||\nabla_k \psi_j(y)||\nabla_k \psi(y)|\dd y\notag\\
 &\qquad\qquad\leq -\gamma_0\sum_{j\in \J}\int_{I_j} |U'(y)||\nabla_k \psi_j(y)|^2\dd y+C\sum_{j\in \J}\int_{I_j}|U'(y)||\nabla_k \psi_j(y)||\nabla_k \psi(y)|\dd y\notag\\
 &\qquad\qquad\leq C\sum_{j\in \J}\int_{I_j} |U'(y)||\nabla_k \psi_j(y)|^2\dd y +\frac{\gamma_{0}}2 \int_{I} |U'(y)| |\nabla_k \psi(y)|^{2}\dd y.
\end{align}
Hence, collecting all of the above and applying Lemma \ref{lem:weighted_elliptic} once more we find
\begin{align}
\frac{\gamma_{0}}2\int_I |U'(y)||\nabla_k \psi(y)|^2\dd y 
\leq C\sum_{j\in \J}\int_{I_j} |U'(y)||\nabla_k \psi_j(y)|^2\dd y,
\end{align}
which is what we wanted. The proof is over.
\end{proof}

With these results at hand, we can now prove the following localized energy inequality, holding for a fairly general collection
of symmetric of operators.

\begin{proposition}\label{prop:reduction}
Let $A(t):L^2\to L^2$ be the symmetric operator defined by
\begin{align}
A(t)= \sum_{j\in\J} \chi_j A_j(t)\chi_j, 
\end{align}
with $A_j$ as in \eqref{eq:Ajs} and  localized potential as in 
\eqref{eq:local1}-\eqref{eq:local2}.   Then it holds that 
  \begin{align}
   \ddt \l \omega, A(t)\omega  \r \leq \Re\l \omega,  \dot A(t)\omega\r +C\eps_0\sum_{j\in \J} \int_{I_j} |U'(y)| \left[|\nabla_k  \psi_j(y)|^2 +k^2|\nabla_k \psi_{j,A}(y) |^2\right]\dd y.   
  \end{align}
\end{proposition}

\begin{proof}
By direct computations and using \eqref{eq:Eulerk}, we find that
\begin{align}
\ddt \l \omega, A(t)\omega  \r
&=\Re\l \omega, \dot A(t)\omega  \r+2\Re\l ik B  \psi , A(t)\omega  \r-2\Re\l ik U \omega, A(t)\omega  \r\notag\\
&=\Re\l \omega, \dot A(t)\omega  \r+2\Re\l ik B  \psi , A(t)\omega  \r,
\end{align}
thanks to the symmetry of $A$ and the anti-symmetry of $ik U$. Now, integrating by parts and using \eqref{eq:local1} and \eqref{eq:local2}, we have
\begin{align}
\l ikB \psi , A(t)\omega  \r&=\sum_{j\in\J} \l ik B  \psi , \chi_j A_j(t)\chi_j \omega  \r=-\sum_{j\in\J} \l ikB  \psi , \chi_j  \Delta_k \psi_{j,A}\r\notag\\
&=k^3\sum_{j\in\J} \l i B  \psi , \chi_j  \psi_{j,A}  \r+k\sum_{j\in\J} \l i\de_y(\chi_j  B  \psi) , \de_{y} \psi_{j,A}  \r.
\end{align}
Thus,  we obtain 
\begin{align}
|\l i kB  \psi , A(t)\omega  \r|
&\leq k^3\sum_{j\in\J} \l i B  \psi , \chi_j  \psi_{j,A}  \r+k\sum_{j\in\J} \l i\de_y(\chi_j  B  \psi) , \de_{y} \psi_{j,A}  \r\notag\\
&\leq \sum_{j\in \J}\bigg[k^3\int_I |B(y)|  |\psi(y)| |\psi_{j,A}(y) |  \chi_j(y)\dd y+ k\int_I |B(y)| |\de_y (\chi_j(y) \psi(y))||\de_y \psi_{j,A}(y)|\dd y \notag\\
&\qquad\quad+ k\int_I |B'(y)| |  \psi(y)|\de_y \psi_{j,A}(y) |\chi_j(y)\dd y\bigg].
\end{align}
We now use \eqref{eq:part1}, \eqref{eq:mildy2} and Lemma \ref{lem:local_elliptic} to deduce that
  \begin{align}
   \ddt \l \omega, A(t)\omega  \r &\leq \Re\l \omega,  \dot A(t)\omega\r +C \sum_{j\in\J}\int_{I_j} \left[|B(y)|+|B'(y)|\right] \left[|\nabla_k  \psi(y)|^2 +k^2|\nabla_k \psi_{j,A}(y) |^2\right]\dd y\notag \\   
   &\leq \Re\l \omega,  \dot A(t)\omega\r +C\eps_0 \sum_{j\in\J}\int_{I_j} |U'(y)| \left[|\nabla_k  \psi(y)|^2 +k^2|\nabla_k \psi_{j,A}(y) |^2\right]\dd y\notag \\   
   &\leq \Re\l \omega,  \dot A(t)\omega\r +C\eps_0 \sum_{j\in\J}\int_{I_j} |U'(y)| \left[|\nabla_k  \psi_j(y)|^2 +k^2|\nabla_k \psi_{j,A}(y) |^2\right]\dd y,
  \end{align}  
and we are done.
\end{proof}

\subsection{The \texorpdfstring{$L^2$}{L2} stability theorem}
Using the reductions of Proposition \ref{prop:reduction}, the proof of Theorem
\ref{thm:L2} reduces to constructing suitable operators $A_{j}(t)$ such that
$\l \omega, A(t)\omega \r$ is a Lyapunov functional. Since on each $I_{j}$ the function $U(y)$ is 
bilipschitz, we can construct adapted Fourier multipliers (strictly speaking just multipliers in a convenient
$L^{2}$ basis, c.f. \cite{Zill5} for other bases). We rely on the following result from \cite{Zill5}.

\begin{lemma}[Bilipschitz case, c.f. \cite{Zill5}]\label{lem:Bilipschitz}
  Suppose that $U$ is bilipschitz on the interval $I_{j}$. There exists
  $A_{j}(t)$ such that 
  \begin{align}
\int_{I_j} |U'(y)| \left[|\nabla_k  \psi_j(y)|^2 +k^2|\nabla_k \psi_{j,A}(y) |^2\right]\dd y &\leq -C \Re\l \chi_{j} \omega, \dot A_{j}(t)\chi_{j}\omega \r, 
  \end{align}
and
  \begin{align}
     \frac1C\|\chi_{j}\omega\|_{L^{2}}^{2}\leq \l \chi_{j} \omega, A_{j}(t)\chi_{j}\omega \r \leq C \|\chi_{j}\omega\|_{L^{2}}^{2}.
  \end{align}  
  Here, $C$  depends on $U$ only via $\frac{\max(|U'|)}{\min(|U'|)}$, and it is therefore independent of $j$. In particular,
    \begin{align}\label{eq:derbound}
\sum_{j\in\J}\int_{I_j} |U'(y)| \left[|\nabla_k  \psi_j(y)|^2 +k^2|\nabla_k \psi_{j,A}(y) |^2\right]\dd y &\leq -C \sum_{j} \Re\l  \chi_{j}\omega, \dot A_{j}(t) \chi_{j}\omega \r.
  \end{align}
\end{lemma}

\begin{proof}
We introduce the new time variable $\tau=t \min(|U'|)$ and observe that \eqref{eq:Eulerkscat} is then given by
\begin{align}
\begin{cases}
\displaystyle\de_{\tau}F  =ik\frac{B(y)}{\min (|U'|)}\Psi,   \quad &\mbox{in } I, \\
\displaystyle \left(-k^2+\left(\de_y-ikt\frac{U'(y)}{\min (|U'|)}\right)^2\right)\Psi  = F, &\mbox{in } I,\\
\Psi= 0, &\mbox{on } \de I.
\end{cases}
\end{align}
We may thus interpret this as a new equation of type \eqref{eq:Eulerkscat} with
  $U,B$ replaced by
 \begin{align}
 U_{\star}:=\frac{U}{\min(|U'|)},\qquad  B_{\star}:=\frac{B}{\min(|U'|)},
  \end{align}
  respectively.
  The flow $U_{\star}$ is then bilipschitz with constants $1$ and
  $\frac{\max(|U'|)}{\min(|U'|)}<C$ by condition \ref{H3}.
  We may hence employ a change of variables $z=U_{\star}(y)$ and define a
  Fourier multiplier with respect to $z$ as in \cite{Zill5}, where the case of a
  bilipschitz profile was analyzed in detail.
\end{proof}

With this preparation, we can now prove our first main result.

\begin{proof}[Proof of Theorem \ref{thm:L2}]
  In view of \ref{H3},  $U$ is Bilipschitz on each interval $I_j$.
  We thus introduce $A_j(t)$ as given in Lemma \ref{lem:Bilipschitz} and define
  \begin{align}
    A(t)= \sum_j \chi_j A_j(t) \chi_j.
  \end{align}
  Then by Proposition \ref{prop:reduction} it holds that
  \begin{align}
    \ddt \l \omega, A(t) \omega\r \leq \Re\l \omega, \dot A (t) \omega \r +C\eps_0\sum_{j\in \J} \int_{I_j} |U'(y)| \left[|\nabla_k  \psi_j(y)|^2 +k^2|\nabla_k \psi_{j,A}(y) |^2\right]\dd y.
  \end{align}
Exploiting \eqref{eq:derbound}, it follows that
  \begin{align}
     \ddt \l \omega, A(t) \omega \r \leq (C\eps_0-1)\sum_{j\in \J}\int_{I_j} |U'(y)| \left[|\nabla_k  \psi_j(y)|^2 +k^2|\nabla_k \psi_{j,A}(y) |^2\right]\dd y. \end{align}
Therefore, given \eqref{eq:mildy2} and \ref{P}, we infer that
  \begin{align}
     \ddt \l \omega, A(t) \omega \r + \frac12\sum_{j\in \J}\int_{I_j} |U'(y)| \left[|\nabla_k  \psi_j(y)|^2 +k^2|\nabla_k \psi_{j,A}(y) |^2\right]\dd y\leq0.
     \end{align}
From this, we deduce that $t\mapsto \l \omega(t), A(t) \omega(t) \r$ is non-increasing.
Furthermore, since $0 \leq \l \omega(t), A(t)  \omega(t) \r$ for all $t\geq 0$,
this also implies that
  \begin{align}
    \int_{0}^T\left[\sum_{j\in \J}\int_{I_j} |U'(y)| \left[|\nabla_k  \psi_j(y)|^2 +k^2|\nabla_k \psi_{j,A}(y) |^2\right]\dd y \right]\dd t \leq \l \omega(0), A(0) \omega(0) \r
  \end{align}
  for all $T\geq 0$. As the left-hand side is a bounded increasing function, we can take the limit as $T\to\infty$ and use
  Lemma \ref{lem:local_elliptic} to deduce that 
  \begin{align}
    \int_{0}^\infty\int_{I} |U'(y)| |\nabla_k  \psi(t,y)|^2 \dd y \dd t \leq \l \omega(0), A(0) \omega(0) \r,
  \end{align}  
  and conclude that proof.
\end{proof}

\section{Boundary layers and \texorpdfstring{$H^1$}{H1} stability}\label{sec:higher}
Building on the results of the previous section, we extend our result to also
establish $H^{1}$ stability and prove the intermediate part of Theorem \ref{thm:L2}.
Due to the potentially degenerate behavior of (derivatives of) the coefficient
functions $U$ and $B$, here we rely on the scattering formulation
\eqref{eq:Eulerkscat} and the differential operator $\dg$ in \eqref{eq:newdiff}
to derive sufficient control of commutator terms.
Here, principal challenges arise from several sources:
\begin{itemize}
\item Taking derivatives of the evolution equation \eqref{eq:Eulerkscat}, we
  obtain commutator terms involving derivatives of the coefficient functions. As
  these coefficient functions are mildly degenerate, we need to establish
  localized estimates controlling derivatives of these functions.
\item While the stream function satisfies zero Dirichlet conditions, this does
  not hold for its derivative. Hence, a boundary layer forms that develops a
  logarithmic singularity near the boundary as time tends to infinity, which was
  studied in \cite{Zill6} and \cite{Zhang2015inviscid}.
  Here, additionally the degeneracy of the coefficients and the need for localized
  estimates (with associated boundary conditions) pose strong technical
  challenges.
\item Furthermore, even in the setting without boundary, due to the degeneracy
  of the coefficient functions, the derivative of the Biot-Savart law introduces
  several non-small commutators.
  We hence adapt these into a modified elliptic operator (c.f. Lemma
  \ref{lem:H1equation}), for which in turn estimates have to be developed.
  This additional modification further necessitates a longer argument in Section
  \ref{sec:H2_2}, where need to take these changes into account in a recursive argument.
\end{itemize}

\subsection{Auxiliary functions and related equations}
\label{sec:auxil-funct}
We begin by studying the equations for $\dg F$. For this, we define the bounded coefficient functions (related to $\dc$ in 
\eqref{eq:newdiff})
\begin{align}
\dc^{(1)}= \sum_{j} \frac{\max_{I_{j}}|U'|}{U'(y)}\chi_{j}',\qquad \dc^{(2)}= \sum_{j} \frac{\max_{I_{j}}|U'|}{U'(y)}\chi_{j}''.
\end{align}
By our choice of $I_{j}$ in terms of level sets of $U'$, $\dc^{(1)}, \dc^{(2)}$
can be estimates in terms of $\epsilon_{0}$ and $\|\chi_{j}'\|_{L^{\infty}}$ and
$\|\chi_{j}''\|_{L^{\infty}}$, respectively.
Taking into account the scaling of this norms in terms of the interval size
$|I_{j}|$ it follows by conditions \ref{P} that
\begin{align*}
  |\dc^{(1)}|+ |\dc^{(2)}| \leq 2 C \epsilon_{0} (1+1/\kappa)^{2} \leq \frac{\gamma_{0}k^{2}}{2},
\end{align*}
which we use in our elliptic estimates to control these contributions as error terms.

In the case $I$ is not the whole space but has boundary points $\{a,b\}=\p I$
(or just a single boundary point),
we further introduce the auxiliary homogenous solutions  $h_{\a}, h_{\b}, \tilde{h}_{\a}, \tilde{h}_{\b}$ of the problems
\begin{align}
&\left(-k^{2}+\de_{yy} -2\de_{y} \frac{\dc^{(1)}}{\dc} -\frac{\dc^{(2)}}{\dc} \right)\tilde{h}_\bullet=0,\label{eq:homo1} \\
&\left(-k^{2}+\de_{yy} \right)h_\bullet=0,\label{eq:homo2}
\end{align}
where $\bullet=\a,\b$, with boundary conditions
\begin{align}
h_{\a}(\a)=\tilde{h}_{\a}(\a)=h_{\b}(\b)=\tilde{h}_{\b}(\b)&=1, \label{eq:boundh1}\\
h_{\a}(\b)=\tilde{h}_{\a}(\b)=h_{b}(\a)=\tilde{h}_{\b}(\a)&=0.\label{eq:boundh2}
\end{align} 
In particular,
\begin{align}
&\left(-k^{2}+(\de_{y}-iktU')^{2} -2(\de_{y}-iktU') \frac{\dc^{(1)}}{\dc} -\frac{\dc^{(2)}}{\dc} \right)\e^{iktU}\tilde{h}_\bullet=0,\label{eq:homo3} \\
&\left(-k^{2}+(\de_{y}-iktU')^{2} \right)\e^{iktU}h_\bullet=0.\label{eq:homo4}
 \end{align}
We then have the following result.
\begin{lemma}\label{lem:H1equation}
Let $F, \Psi$ be the solution to \eqref{eq:Eulerkscat}.
Then $\dg F$ satisfies the following equations, where the terms past the first
line only appear
in the setting with boundary:
\begin{align}\label{eq:DyF}
\de_{t} \dg F &=  ik B \tPhi + (\dg B)ik \Psi\notag \\
&\quad+ B \dc(\a)\l  \dg F,\frac{ \e^{ikt(U(y)-U(\a))}}{tU'd} h_\a  \r  \e^{{ikt(U(y)-U(\a))}}\tilde{h}_{\a}\notag \\
&\quad-  B\dc(\b)\l  \dg F,\frac{ \e^{ikt(U(y)-U(\b))}}{tU'd} h_\b  \r    \e^{{ikt(U(y)-U(\b))}}\tilde{h}_{\b}\notag\\
&\quad+  B \dc(\a)\l  \de_y\left(\frac{h_\a}{tU'}\right)F,\e^{ikt(U(y)-U(\a))}  \r  \e^{{ikt(U(y)-U(\a))}}\tilde{h}_{\a}\notag \\
&\quad-  B\dc(\b)\l  \de_y\left(\frac{h_\b}{tU'}\right)F,\e^{ikt(U(y)-U(\b))}  \r      \e^{{ikt(U(y)-U(\b))}}\tilde{h}_{\b}\notag\\
&\quad+   \frac{\dc\omega^{in}}{tU'}\bigg|_{y=\a} B\e^{{ikt(U(y)-U(\a))}}\tilde{h}_{\a}+  \frac{\dc\omega^{in}}{tU'}\bigg|_{y=\b}B   \e^{{ikt(U(y)-U(\b))}}\tilde{h}_{\b}
\end{align}
where  
\begin{align}
&\left(-k^{2}+(\de_{y}-iktU')^{2}-2 (\de_{y}-iktU') \frac{\dc^{(1)}}{\dc}-\frac{\dc^{(2)}}{d}\right) \tPhi \notag\\
&\qquad\qquad= \dg F +\left[ \left(\frac{U''}{U'}\dc\right)'- \frac{U''}{U'}\dc^{(1)}- 2 (\de_{y}-iktU') \frac{U''}{U'}\dc\right]\left(\de_y-iktU'\right)\Psi,
\end{align}   
with boundary conditions $\tPhi |_{y=a,b} =0$.
\end{lemma}

\begin{proof}
As a first step, we apply $\dg$ to \eqref{eq:Eulerkscat} and obtain
\begin{align}
&\de_t \dg F = ikB \dg \Psi + ik (\dg B)\Psi,\label{eq:eqforDF} \\
&\left(-k^{2}+(\de_{y}-iktU')^{2}\right)\dg \Psi = \dg F + [(\de_{y}-iktU')^{2}, \dg] \Psi, \label{eq:eqforDPsi}
\end{align}
Using the commutator relation $[A^2,B]=2A[A,B]+[[A,B],A]$, and that
\begin{align}\label{eq:comm0}
[\de_y-iktU',\dg]=\dc'\de_y+iktU''\dc=\dc^{(1)}\de_y- \frac{U''}{U'}\dc\left(\de_y-iktU'\right)
\end{align}
we compute
\begin{align}\label{eq:comm1}
[(\de_{y}-iktU')^{2}, \dg] \Psi &= 2 (\de_{y}-iktU')[\de_{y}-iktU', \dg] \Psi + [\de_{y}-iktU',[\de_{y}-iktU', \dg]] \Psi \notag\\
    &=- 2 (\de_{y}-iktU') \frac{U''}{U'}\dc\left(\de_y-iktU'\right)\Psi+ 2 (\de_{y}-iktU') \dc^{(1)} \de_{y} \Psi \notag\\
    &\quad+ \left[(\de_{y}-iktU'), \frac{U''}{U'}\dc\left(\de_y-iktU'\right)\right]\Psi + [(\de_{y}-iktU'),\dc^{(1)} \de_{y}]\Psi. 
\end{align}
The first term is already in the desired form. Also, we simply rewrite the second term as 
\begin{align}\label{eq:comm2}
2 (\de_{y}-iktU') \dc^{(1)} \de_{y} \Psi =2 (\de_{y}-iktU') \frac{\dc^{(1)}}{\dc} \dg \Psi
\end{align}
For the third term, we note that the operator $\de_{y}-iktU'$ commutes
with itself and we hence obtain
\begin{align}\label{eq:comm3}
\left[(\de_{y}-iktU'), \frac{U''}{U'}\dc\left(\de_y-iktU'\right)\right]\Psi 
&=\left[(\de_{y}-iktU'), \frac{U''}{U'}\dc\right]\left(\de_y-iktU'\right)\Psi\notag\\
&= \left(\frac{U''}{U'}\dc\right)'\left(\de_y-iktU'\right)\Psi,
\end{align}
while for the fourth term we directly compute that
\begin{align}\label{eq:comm4}
[(\de_{y}-iktU'),\dc^{(1)} \de_{y} \Psi]
&=\dc^{(2)}\de_y\Psi- \frac{U''}{U'}\dc^{(1)}\left(\de_y-iktU'\right)\Psi\notag\\
&=\frac{\dc^{(2)}}{d}\dg\Psi- \frac{U''}{U'}\dc^{(1)}\left(\de_y-iktU'\right)\Psi.
\end{align}
Collecting all of the above, we obtain 
\begin{align}
&\left(-k^{2}+(\de_{y}-iktU')^{2}-2 (\de_{y}-iktU') \frac{\dc^{(1)}}{\dc}-\frac{\dc^{(2)}}{d}\right)\dg \Psi \notag\\
&\qquad\qquad= \dg F +\left[ \left(\frac{U''}{U'}\dc\right)'- \frac{U''}{U'}\dc^{(1)}- 2 (\de_{y}-iktU') \frac{U''}{U'}\dc\right]\left(\de_y-iktU'\right)\Psi
\end{align}
We then split $\dg \Psi$ into a non-homogenous solution $\tPhi$ with homogeneous Dirichlet boundary conditions and
the linear combination of homogeneous solutions to \eqref{eq:homo3}, namely
\begin{align}\label{eq:homoline}
\dg \Psi=\tPhi+ (\dg\Psi|_{y=\a}) \e^{{ikt(U(y)-U(\a))}}\tilde{h}_{\a} +(\dg\Psi|_{y=b})    \e^{{ikt(U(y)-U(\b))}}\tilde{h}_{\b},
\end{align}
where $\tilde{h}_{\a}$ and $\tilde{h}_{\b}$ are defined in \eqref{eq:homo1}. 
Furthermore, using integration by parts and recalling from \eqref{eq:Eulerkscat} that $\Psi$ vanishes at the boundary, 
we find that for any homogeneous solution $h_\bullet$ we have
\begin{align}
\l F, \e^{iktU}h_\bullet  \r &=   \l (-k^{2}+(\de_{y}-iktU')^{2}) \Psi, \e^{iktU}h_\bullet  \r\notag\\
    &=  \e^{iktU}h_\bullet  (\de_{y}-iktU')\Psi \big|_{y=\a}^{\b} -\e^{iktU} \Psi  (\de_{y}-iktU')  h_\bullet\big|_{y=\a}^{\b}\notag\\
    &=  \e^{iktU}h_\bullet \de_{y}\Psi \big|_{y=\a}^{\b}=\e^{iktU}h_\bullet \frac{1}{d}\dg\Psi \big|_{y=\a}^{\b}.
  \end{align}
Therefore, from \eqref{eq:boundh1}-\eqref{eq:boundh2} we conclude that
\begin{align}\label{eq:boundaryterms}
\dg\Psi|_{y=\a} =-\dc(\a)\l F, \e^{ikt(U(y)-U(\a))}h_\a  \r   , \qquad \dg\Psi|_{y=\b} =\dc(\b)\l F, \e^{ikt(U(y)-U(\b))}h_\b  \r,
\end{align}
Thus, integrating by parts we obtain
\begin{align}
\dg\Psi|_{y=\a} 
&=-\dc(\a)\l  \frac{h_\a}{iktU'}F,\de_y\e^{ikt(U(y)-U(\a))}  \r \notag \\
&=\dc(\a)\l  \de_y\left(\frac{h_\a}{iktU'}\right)F,\e^{ikt(U(y)-U(\a))}  \r  +\dc(\a)\l  \frac{h_\a}{iktU'}\de_yF,\e^{ikt(U(y)-U(\a))}  \r +\frac{\dc(\a)F(\a)}{iktU'(\a)} \notag\\
&=\dc(\a)\l  \de_y\left(\frac{h_\a}{iktU'}\right)F,\e^{ikt(U(y)-U(\a))}  \r  +\dc(\a)\l  \dg F,\frac{ \e^{ikt(U(y)-U(\a))}}{iktU'd} h_\a  \r +\frac{\dc(\a)F(\a)}{iktU'(\a)}.
\end{align}
In a similar fashion,
\begin{align}
\dg\Psi|_{y=\b} 
=-\dc(\b)\l  \de_y\left(\frac{h_\b}{iktU'}\right)F,\e^{ikt(U(y)-U(\b))}  \r  -\dc(\b)\l  \dg F,\frac{ \e^{ikt(U(y)-U(\b))}}{iktU'd} h_\b  \r +\frac{\dc(\b)F(\b)}{iktU'(\b)}.
\end{align}
Going back to \eqref{eq:eqforDF}, we use  \eqref{eq:homoline} and the above computations to obtain precisely \eqref{eq:DyF}, upon
using that the boundary of $F(t)$ is preserved (c.f. Remark \ref{remark:boundary_data}),
and thus finishing the proof. 
\end{proof}

Having established the modified equation, we show that we can localize estimates
as in Section \ref{sec:simple_case}. Here, as in \cite{Zill6} we further split
$\dg F$ into contributions $\beta_{\a}, \beta_{\b}$ with zero initial data and right-hand-side involving
$\omega^{in}(\a),\omega^{in}(\b)$ and another contribution $F^{(1)}$ with initial data $\dg F|_{t=0}$
but simpler right-hand-side.
This splitting allows us to separately treat the different time behavior and
growth of boundary terms and does not appear in the setting without boundary.

We state the splitting in the following lemma, whose proof is only based on
linearity of the system considered. 
\begin{lemma}
  \label{lem:splitting}
  In the setting and notation of Lemma \ref{lem:H1equation}, it holds that
  we write $\dg F=\beta_{\a}+ \beta_{\b}+F^{(1)}$,
  where $F^{(1)}$ is the unique solution of
  \begin{align}\label{eq:DyF1}
\de_{t} F^{(1)} &=  ik B \tPsi + (\dg B)ik \Psi \notag\\
&\quad+ B \dc(\a)\l  F^{(1)},\frac{ \e^{ikt(U(y)-U(\a))}}{tU'd} h_\a  \r  \e^{{ikt(U(y)-U(\a))}}\tilde{h}_{\a} \notag\\
&\quad-  B\dc(\b)\l  F^{(1)},\frac{ \e^{ikt(U(y)-U(\b))}}{tU'd} h_\b  \r    \e^{{ikt(U(y)-U(\b))}}\tilde{h}_{\b}\notag\\
&\quad+  B \dc(\a)\l  \de_y\left(\frac{h_\a}{tU'}\right)F,\e^{ikt(U(y)-U(\a))}  \r  \e^{{ikt(U(y)-U(\a))}}\tilde{h}_{\a}\notag \\
&\quad-  B\dc(\b)\l  \de_y\left(\frac{h_\b}{tU'}\right)F,\e^{ikt(U(y)-U(\b))}  \r      \e^{{ikt(U(y)-U(\b))}}\tilde{h}_{\b},
\end{align}
with initial condition $F^{(1)}|_{t=0}=\dg F|_{t=0}$, together with 
\begin{align}\label{eq:DyF2}
&\left(-k^{2}+(\de_{y}-iktU')^{2}-2 (\de_{y}-iktU') \frac{\dc^{(1)}}{\dc}-\frac{\dc^{(2)}}{d}\right) \tPsi \notag\\
&\qquad\qquad= F^{(1)} +\left[\de_{y} \left(\frac{U''}{U'}\dc\right)- \frac{U''}{U'}\dc^{(1)}- 2 (\de_{y}-iktU') \frac{U''}{U'}\dc\right]\left(\de_y-iktU'\right)\Psi,
\end{align}   
with boundary conditions $\tPsi|_{y=\a,\b}=0$.
Similarly, $\beta_\bullet$ is the unique solution of
\begin{align}\label{eq:beta}
\de_{t}  \beta_\bullet &= ik B \tPsi_\bullet+B \dc(\a)\l  \beta_\bullet,\frac{ \e^{ikt(U(y)-U(\a))}}{tU'd} h_\a  \r  \e^{{ikt(U(y)-U(\a))}}\tilde{h}_{\a}\notag \\
&\quad-  B\dc(\b)\l  \beta_\bullet,\frac{ \e^{ikt(U(y)-U(\b))}}{tU'd} h_\b  \r    \e^{{ikt(U(y)-U(\b))}}\tilde{h}_{\b}
+    \frac{\dc\omega^{in}}{tU'}\bigg|_{y=\bullet} B\e^{{ikt(U(y)-U(\bullet))}}\tilde{h}_{\bullet},
\end{align}
for $\bullet=a,b$, with initial condition $\beta_\bullet|_{t=0}=0$,  together with  
\begin{align}\label{eq:betapsi}
&\left(-k^{2}+(\de_{y}-iktU')^{2}-2 (\de_{y}-iktU') \frac{\dc^{(1)}}{\dc}-\frac{\dc^{(2)}}{d}\right) \tPsi_\bullet= \beta_\bullet,
\end{align}   
and boundary conditions $\tPsi_\bullet|_{y=\a,\b}=0$.
  
\end{lemma}

\subsection{Localization and estimates}

Following the notational conventions of Section \ref{sub:localizedpote}, 
we introduce $\tPhi_{j}$ and $\tPhi_{j,A}$ for the solutions of
\begin{equation}\label{eq:local3}
\begin{cases}
(-k^2+(\de_{y}-iktU')^{2})\tPhi_j = \chi_j F^{(1)}, \quad &\mbox{in } I_j , \\
\tPhi_{j}= 0, &\mbox{on } \de I_j, 
\end{cases}
\end{equation}
and
\begin{equation}\label{eq:local4}
\begin{cases}
 (-k^2+(\de_{y}-iktU')^{2}) \tPhi_{j,A} = A_j\chi_j F^{(1)}, \quad &\mbox{in }  I_j , \\
\tPhi_{j,A}= 0, &\mbox{on }  \de I_j,
\end{cases}
\end{equation}
We introduce the short-hand-notation 
\begin{align}
\nabla_{k,t}:=(ik, \de_{y}-iktU')
\end{align}
along with an associated Hilbert-space $H^1_t$ endowed with scalar product
\begin{align}\label{eq:4}
\l g_1,g_2 \r_{H^1_t}= \l g_1,g_2 \r + \l \nabla_{k,t} g_1, \nabla_{k,t} g_2 \r, \qquad \| g \|^2_{H^1_t}= \|g\|^2 + \| \nabla_{k,t} g\|^2.
\end{align}
Similarly, for any weight function $a(y)$, $H^1_t(a)$ refers to the space defined using $L^2(a \dd x \dd y)$ instead.

\begin{lemma}\label{lem:H1elliptic}
In the setting of Lemma \ref{lem:H1equation} additionally suppose that \ref{P}
is satisfied.
Then for any interval $J$ and any given function $g \in L^{2}(J)$, the unique solution $\psi$ of 
\begin{align}
\left(-k^{2}+(\de_{y}-iktU')^{2}-2 (\de_{y}-iktU') \frac{\dc^{(1)}}{\dc}-\frac{\dc^{(2)}}{d}\right)\psi &= g, \label{eq:H1ell1}\\
\psi|_{\de J}&=0,
\end{align}
and the unique solution $\psi^{(1)}$ of
\begin{align}
(-k^2+(\de_{y}-iktU'(y))^{2})\psi^{(1)}&= g,\label{eq:H1ell2} \\
\psi^{(1)}|_{\de J}&=0,
\end{align}
satisfy
\begin{align}
\int_{J} |U'| |\nabla_{k,t}\psi|^{2} \leq C \int_{J} |U'| |\nabla_{k,t}\psi^{(1)}|^{2}. 
\end{align}
\end{lemma}

\begin{proof}
We test \eqref{eq:H1ell1} with $-|U'| \psi$ and obtain that
\begin{align}
-\l |U'|  \psi, g \r=\l \nabla_{k,t} \psi,\nabla_{k,t}(|U'|\psi)\r +\l\left(2 (\de_{y}-iktU') \frac{\dc^{(1)}}{\dc}+\frac{\dc^{(2)}}{\dc}\right)\psi,|U'|\psi \r
\end{align}
Now, arguing as in the proof of Lemma \ref{lem:weighted_elliptic} and using that 
$U'\de_y|U'|=|U'|U''$, we have that 
\begin{align}
\Re\l \nabla_{k,t} \psi,\nabla_{k,t}(|U'|\psi)\r
&=\int_J |U'(y)| |\nabla_{k,t}\psi(y)|^{2}\dd y+\Re\l (\de_y-iktU') \psi,\psi \de_y|U'|)\r\notag\\
&=\int_J |U'(y)| |\nabla_{k,t}\psi(y)|^{2}\dd y+\frac12\l \de_y |\psi|^2, \de_y|U'|)\r-\Re\l iktU' \psi,\psi \de_y|U'|)\r\notag\\
&=\int_J |U'(y)| |\nabla_{k,t}\psi(y)|^{2}\dd y-\frac12\l  |\psi|^2, \de_{yy}|U'|)\r\notag\\
&=\int_J |U'(y)| |\nabla_{k,t}\psi(y)|^{2}\dd y-\frac12\int_J  \sign(U'(y))U'''(y) |\psi(y)|^2\dd y\notag\\
&\geq\left(1-\frac{1-\gamma_0}{k^2}\right)\int_J |U'(y)||\nabla_{k,t} \psi(y)|^2\dd y.
\end{align}
Moreover, in light of the relations 
\begin{align}
\de_y \dc=\dc^{(1)}- \frac{U''}{U'}\dc, \qquad \de_y \dc^{(1)}=\dc^{(2)}- \frac{U''}{U'}\dc^{(1)},
\end{align}
we deduce that
\begin{align}
\l(\de_{y}-iktU') \frac{\dc^{(1)}}{\dc}\psi,|U'|\psi \r
&=\l\frac{\dc^{(1)}}{\dc}(\de_{y}-iktU') \psi,|U'|\psi \r+\int_J\left[\frac{\dc^{(2)}}{\dc}-\left(\frac{\dc^{(1)}}{\dc}\right)^2\right]|U'(y)||\psi(y)|^2\dd y
\end{align}
In view on the assumptions on $\dc^{(1)}$ and $\dc^{(2)}$ and the above observation, we have
\begin{align}
|\l(\de_{y}-iktU') \frac{\dc^{(1)}}{\dc}\psi,|U'|\psi \r|\leq \frac{C}{k}\int_J |U'(y)||\nabla_{k,t} \psi(y)|^2\dd y.
\end{align}
Similarly, 
\begin{align}
|\l\frac{\dc^{(2)}}{\dc}\psi,|U'|\psi \r|\leq \frac{C}{k}\int_J |U'(y)||\nabla_{k,t} \psi(y)|^2\dd y.
\end{align}
Hence, using that $\gamma_{0}-\frac{2C}{k} \geq \frac{\gamma_{0}}{2}$
\begin{align}\label{eq:loaef}
\int_J |U'(y)||\nabla_{k,t} \psi(y)|^2\dd y \leq 2\Re\l \nabla_{k,t} \psi,\nabla_{k,t}(|U'|\psi)\r
\end{align}
On the other hand, using \eqref{eq:H1ell2} we find that
\begin{align}
-\l |U'|  \psi, g \r&=-\l|U'|\psi , (-k^2+(\de_{y}-iktU'(y))^{2})\psi^{(1)}\r\notag\\
&=\l|U'|\nabla_{k,t}\psi , \nabla_{k,t}\psi^{(1)}\r+\l\psi\de_y|U'| ,(\de_{y}-iktU')\psi^{(1)}\r.
\end{align}
Therefore, since $\de_y|U'|=\sign(U') U''$, we find that
\begin{align}
|\l |U'|  \psi, g \r|
&\leq \int_J |U'(y)||\nabla_{k,t} \psi(y)||\nabla_{k,t} \psi^{(1)}(y)|\dd y+ \int_J |U''(y)|| \psi(y)||\nabla_{k,t} \psi^{(1)}(y)|\dd y\notag\\
&\leq C\int_J |U'(y)||\nabla_{k,t} \psi(y)||\nabla_{k,t} \psi^{(1)}(y)|\dd y.
\end{align}
The claim thus follows from the above inequality and \eqref{eq:loaef}, by using H\"older's and Young's inequality.
 \end{proof}

In the case with boundary, we further need to control the boundary corrections
in Lemma \ref{lem:splitting}.

\begin{lemma}[c.f. \cite{Zill6}]
  \label{lem:boundaryCorrection}
  In the setting and notation of Lemma \ref{lem:H1equation}, let
  $I_{j}=(\a_{j},\b_{j})$ and $g
  \in L^{2}(I_{j})$ be any given function and let $(\chi_{j}g)_{n}$ denote the Fourier
  basis expansion on the interval $I_{j}$ with respect to the
  variable $z=\frac{U(y)-U(\a_{j})}{\min_{I_j} |U'|}$.
  Then, it holds that 
  \begin{align}
    \l g, ikB \e^{ikt(U(y)-U(\a))}\tilde{h}_{\a}\r &\leq C \sum_{j,n} \frac{|k(\chi_{j}g)_{n}|\|B\chi_{j} \tilde{h}_{\a}\|_{H^{1}} }{|k|+|n-kt\min(U')|}, \\
    \langle g, \frac{\e^{ikt(U(y)-U(\a))}}{iktU' \dc}h_{\a}\rangle &\leq C \sum_{j,n} \frac{|(\chi_{j}g)_{n}|\|\chi_{j} \tilde{h}_{\a}\|_{H^{1}}}{(|k|+|n-kt\min(U')|) |kt\min U'|}.
  \end{align}
  In particular, we can further estimate
  \begin{align}
    \sum_{j,n} \frac{|k(\chi_{j}g)_{n}|\|B\chi_{j} \tilde{h}_{\a}\|_{H^{1}} }{|k|+|n-kt\min(U')|}
    \leq \left\|\frac{B\tilde{h}_{\a}}{U'}\right\|_{H^{1}} C_{\delta} \left( \sum_{j,n} \frac{|(\chi_{j}g)_{n}|^{2}|\min U'|}{(1+|n/k-t\min U'|)^{1-\delta}} \right)     
  \end{align}
\end{lemma}

\begin{proof}
  In what follows, denote by $c_j=\min_{I_j}|U'|$.
  For convenience of notation, we may further without loss of generality assume
  that $U'>0$ on $I_{j}$.
  Expressed in terms of $y$, an orthonormal Fourier basis of $L^{2}(I_{j},\frac{U'}{c}
  dy)$ is given by
  \begin{align}
    e_{n}(y):=\frac{1}{\|\frac{U'}{c}\|_{L^{2}(I_{j},dy)}} \exp( in \frac{U(y)-U(a_{j})}{c} \frac{2\pi}{\frac{U(b_{j})-U(a_{j})}{c_{j}}}),
  \end{align}
  where $n \in \mathbb{Z}$.
  We note that, by condition \ref{H3} and the definition of
  $I_{j}$, $\frac{U'}{c_{j}}$ is comparable to $1$ and the $L^{2}$ normalizing
  factor $\sqrt{\frac{c_{j}}{|U(b_{j})-U(a_{j})|}} \approx
  \frac{1}{\sqrt{I_{j}}} $ is bounded above. Similarly, by the mean value
  theorem
  \begin{align}
 l_{j}:=\frac{2\pi}{\frac{U(b_{j})-U(a_{j})}{c_{j}}} = \frac{2 \pi}{|I_{j}|} \frac{c_{j}}{U'(\tilde{y})} \approx \frac{2 \pi}{|I_{j}|}
  \end{align}
  is bounded.
  Hence, the norms and normalizations with respect to
  $L^{2}(I_{j},dy)$ and $L^{2}(I_{j}, \frac{U'}{c} dy)$ are comparable within a
  uniform factor.
  Recalling our partition of unity $\chi_{j}$, we thus expand 
  \begin{align}
    \l g, ikB \e^{ikt(U(y)-U(\a))}\tilde{h}_{\a}\r = \sum_{j} \l \chi_{j} g , \frac{c}{U'} ikB \e^{ikt(U(y)-U(\a))}\tilde{h}_{\a}\r_{L^{2}(\frac{U'}{c_{j}} dy)} \\
    =  \sum_{j,n} (\chi_{j}g)_{n }\l  e_{n}, \frac{c}{U'} ikB \e^{ikt(U(y)-U(\a))}\tilde{h}_{\a}\r_{L^{2}(\frac{U'}{c_{j}} dy)} \\
    =  \sum_{j,n} (\chi_{j}g)_{n }\l  e_{n}, ikB \e^{ikt(U(y)-U(\a))}\tilde{h}_{\a}\r_{L^{2}(dy)},
  \end{align}
  and similarly for $\langle g, \frac{\e^{ikt(U(y)-U(\a))}}{iktU'
    \dc}h_{\a}\rangle$.
  By the above considerations, it thus suffices to estimate each summand for
  fixed $j$ and $n$.
  For convenience of notation we here rescaled $n$ by $l_{j}$. As noted above,
  $l_{j}$ is bounded above and upon further partitioning the intervals $I_{j}$,
  we may also assume $l_{j}$ is bounded away from 0. (If $|I_{j}|$ is large, we
  may use the $L^{2}$ normalization when summing in $n l_{j}$. A further
  partitioning is thus not necessary, but merely notationaly convenient.)
  It thus remains to show that
  \begin{align}
\l \e^{in\frac{U(y)-U(\a_j)}{c_j}}, ikB\chi_j \e^{ikt(U(y)-U(\a))}\tilde{h}_{\a}\r & \leq C\frac{|k|\|B \chi_{j} \tilde{h}_{a}\|_{H^{1}} }{|k|+|n-kt\min U'|},\label{eq:ineq1} \\
\langle  \e^{in\frac{U(y)-U(\a_j)}{c_j}}, \frac{\e^{ikt(U(y)-U(\a))}}{iktU' \dc}h_{\a}\rangle &\leq C\frac{|k|\|\chi_{j} h_{a}\|_{H^{1}}}{(|k|+|n-kt\min U'|) kt\min U'}\label{eq:ineq2} ,
  \end{align}
for a constant $C$ independent of $n$ and $j$.

For \eqref{eq:ineq1}, when $\a_j\neq \a$, an integration by parts yields 
(notice that no boundary term appears due to the presence of $\chi_j$)
\begin{align}
\l  \e^{inz}, ikB\chi_j \e^{ikt(U(y)-U(\a))}\tilde{h}_{\a}\r
&=\l   \e^{inz-ikt(U(y)-U(\a))}, ikB\chi_j \tilde{h}_{\a}\r\notag\\
&=\l \frac{1}{i(n/c_j-kt)U'}\de_y\e^{inz-ikt(U(y)-U(\a))}, ikB\chi_j \tilde{h}_{\a}\r\notag\\
&=-\l \frac{1}{i(n/c_j-kt)U'}\e^{inz-ikt(U(y)-U(\a))}, ik\de_y(B\chi_j \tilde{h}_{\a})\r\notag\\
&\quad-\l \frac{U''}{i(n/c_j-kt)(U')^2}\e^{inz-ikt(U(y)-U(\a))}, ikB\chi_j \tilde{h}_{\a}\r.
\end{align}
Recalling that $|U''|\leq C|U'|$ by \ref{H2}, the definition of $c_j$
and the normalization of the Fourier basis, the last two terms
can be bounded by
\begin{align}
  |\l  \e^{inz}, ikB\chi_j \e^{ikt(U(y)-U(\a))}\tilde{h}_{\a}\r|\leq
  C\frac{k\|B \chi_{j} \tilde{h}_{\a}\|_{H^{1}} }{k+|n-ktc_j|}\leq C\frac{k\|B \chi_{j} \tilde{h}_{\a}\|_{H^{1}} }{k+|n-kt\min U'|}
\end{align}
When $\a_j=\a$, $\chi_j(\a)=1$, and we have the additional boundary term
\begin{align}
\frac{kB(\a)}{(n/c_j-kt)U'(\a)},
\end{align}
which is comparable with the above estimates. Also, \eqref{eq:ineq2} follows by very similar arguments.
For the last estimate, we note that
\begin{align}
\sum_{j} \left\|\frac{B\chi_{j} \tilde{h}_{a}}{\min U'}\right\|_{H^{1}}^{2} \leq C \left\|\frac{B\tilde{h}_{a}}{U'}\right\|_{H^{1}}^{2},
\end{align}
since $\chi_{j}^{2}$ is a partition of unity with
  $\|\chi_{j}\|_{W^{{1,\infty}}}<C$.
  Furthermore, for every $\delta>0$, there exists $C_\delta>0$ such that
  \begin{align}
    \sum_{n} \frac{1}{(1+|n/k-t\min U'|)^{1+\delta}} < C_{\delta}k.
  \end{align}
  The result hence follows by applying the Cauchy-Schwarz inequality in $j$ and $n$.
\end{proof}
With these preparations, we can construct our building block for $A^{(1)}(t)$:
\begin{lemma}[c.f. \cite{Zill6}]
  \label{lem:modifiedMultiplier}
  For each interval $I_{j}$, there exists an operator $A^{(1)}_{j}(t)$ such that for
  any $g \in L^{2}(I)$
  \begin{align}
    \langle \chi_{j}g, A^{(1)}_{j}(t) \chi_{j}g \rangle &\approx \|\chi_{j}g\|_{L^{2}}^{2}, \\
    -\langle \chi_{j}g, \dot A^{(1)}_{j}(t) \chi_{j}g  \rangle &\geq \|\chi_{j}g\|_{H^{{-1}}_{t}}^{2} + \sum_{n} \frac{|(\chi_{j}g)_{n}|^{2}\min U'}{(|k|+|n-kt\min(U')|^{1-\delta}) |kt\min U'|^{1-\delta}}\label{eq:signAdot},
  \end{align}
  where the sum over $n$ again denotes the (rescaled) basis expansion as in the
  previous lemmas.
\end{lemma}

\begin{proof}
  We define $A^{(1)}_{j}(t)$ as the multiplier
  \begin{align}
    \exp\left( \arctan(n-kt \min(U')) + \int_0^{t} \frac{\min U'}{(|k|+|n-k\tau\min(U')|^{1-\delta}) |k\tau\min U'|^{1-\delta}} \dd\tau \right).
  \end{align}

\end{proof}

\begin{proposition}
\label{prop:F1}
  Let $\beta_{\bullet}, F^{(1)}$ be as in Lemma \ref{lem:splitting} and assume the
  conditions of Theorem \ref{thm:L2}.
  Then
  \begin{align}
    \|\beta_{\bullet}(t)\|_{L^{2}} &\leq C \left|\omega^{in}|_{\bullet}\right|, \\
    \|F^{(1)}(t)\|_{L^{2}} &\leq C \|\omega^{in}\|_{H^{1}}.
  \end{align}
\end{proposition}

\begin{proof}
Let $A^{(1)}$ be as in Lemma \ref{lem:modifiedMultiplier}. We begin by treating  $\beta_\a$. From \eqref{eq:beta} it follows
that
\begin{align}
\ddt \langle \beta_\a, A^{(1)}(t) \beta_\a \rangle &=\langle \beta_\a, \dot A^{(1)}(t) \beta_\a \rangle\notag \\
&\quad + 2 \Re \langle  ik B \tPsi_\a ,  A^{(1)}(t) \beta_\a \rangle \notag\\
&\quad + 2 \dc(\a)\Re \langle B   \e^{{ikt(U(y)-U(\a))}}\tilde{h}_{\a},  A^{(1)}(t) \beta_\a \rangle \l  \beta_\a,\frac{ \e^{ikt(U(y)-U(\a))}}{tU'd} h_\a  \r\notag\\
&\quad - 2\dc(\b) \Re \langle  B    \e^{{ikt(U(y)-U(\b))}}\tilde{h}_{\b},  A^{(1)}(t) \beta_\a \rangle\l  \beta_\a,\frac{ \e^{ikt(U(y)-U(\b))}}{tU'd} h_\b  \r \notag\\
&\quad + 2 \frac{\dc(\a)\omega^{in}(\a)}{tU'(\a)}\Re \langle B  \e^{{ikt(U(y)-U(\a))}}\tilde{h}_{\a},  A^{(1)}(t) \beta_\a \rangle 
\end{align}
Following a similar strategy as in Section \ref{sec:simple_case}, we define
localized potentials $\Psi_{\a,A,j}^{(1)}$ as the solutions of
\begin{align*}
  (-k^{2}+(\de_{y}-iktU')^{2}) \Psi_{\a,A,j}^{(1)}&= A_{j}(t)\chi_{j} \beta_{a} \text{ on } I_{j} \\
  \Psi_{\a,A,j}^{(1)}|_{\p I_{j}}=0.
\end{align*}
As $\chi_{j}^{2}$ is a partition of unity, we obtain the identity
  \begin{align}
  A^{(1)}(t) \beta_\a= \sum_{j}\chi_{j} (-k^{2}+(\de_{y}-iktU')^{2}) \Psi_{\a,A,j}^{(1)}
  \end{align}
  and integrate the second term above by parts to estimate
  \begin{align}
    2 \Re \langle Bik \Psi_{i}^{(1)},  A^{(1)}(t) \beta_{i} \rangle & \leq C \int |U'| (k^{2}|\Psi_{i}^{(1)}|^{2} +|(\de_{y}-iktU')\Psi_{i}^{(1)}|^{2} ) \\
   & \quad + \sum_{j}\int |U'| (k^{2}|\Psi_{i,A,j}^{(1)}| +|(\de_{y}-iktU')\Psi_{i,A,j}^{(1)}|^{2}). 
  \end{align}
  We then apply Lemma \ref{lem:H1elliptic} to estimate $\Psi_{i}^{(1)}$ in terms
  of a stream function given by the standard elliptic operator, which in turn is
  estimated by localized stream functions as in Section \ref{sec:simple_case}.

  In order to estimate the scalar product involving $h_{i}$ and
  $\tilde{h}_{i}$ on the second line, we use Lemma \ref{lem:boundaryCorrection}.
  Here, we further note that the additional term $\frac{1}{|kt\min U'|}$ can be
  written as $\frac{1}{|kt|^{\delta}} \frac{|\min U'|^{-\delta}}{|kt\min
    U'|^{1-\delta}}$.
  Our construction of the modified multiplier in Lemma
  \ref{lem:modifiedMultiplier}, was chosen in just such a way that we
  can absorb this contribution using $\langle \beta_{i}, \dot A^{(1)}(t)
  \beta_{i} \rangle$, provided a smallness assumption is satisfied.
  However, this smallness criterion is sure to hold for large times, since
  $\frac{1}{|kt|^{\delta}}$ tends to zero as $t\rightarrow \infty$.

  Finally, for the last contribution due to $\omega^{in}(\bullet)$, we use Young's
  inequality to estimate
  \begin{align}
     2\Re \frac{ikB\omega^{in}(\bullet)}{iktU'(\bullet)} \langle A^{(1)} \beta_{\bullet},\e^{ikt U(y)-U_{\bullet}} \tilde{h}_{\bullet}\rangle \leq \frac{C (B\omega^{in}(\bullet))^{2}}{U'(\bullet)^{2} t^{1+\sigma}} + C \frac{C (B\omega^{in}(\bullet))^{2}}{U'(\bullet)^{2} t^{1-\sigma}} |\langle A^{(1)} \beta_{'bullet},\e^{ikt U(y)-U_{\bullet}} \tilde{h}_{\bullet}\rangle|^{2}.
  \end{align}
  Here, we can choose $\sigma=\delta$ or $\delta< \sigma <1$ so that the second
  term is small and can be absorbed as the previous term.
  
  In summary, we hence obtain that for $t$ sufficiently big
  \begin{align}
    \ddt \langle \beta_{\bullet}, A^{(1)}(t) \beta_{\bullet} \rangle \leq c \langle \beta_{\bullet}, \dot A^{(1)}(t) \beta_{\bullet} \rangle
    + C \frac{B\omega^{in}(\bullet)}{U'(\bullet)} \frac{1}{t^{1+\delta}}
  \end{align}
  and the result hence follows by integration and using Gronwall's lemma for
  small times.

  Similarly, for $F^{(1)}$, we conclude that
  \begin{align}
    \ddt \langle F^{(1)}, A^{(1)}(t) F^{(1)} \rangle \leq  c \langle F^{(1)}, \dot A^{(1)}(t) F^{(1)} \rangle
    + C  \frac{\|F(t)\|_{L^{2}}}{t^{1+\delta}} \leq c \langle F^{(1)}, \dot A^{(1)}(t) F^{(1)} \rangle + C \frac{\|\omega^{in}\|_{L^{2}}}{t^{1+\delta}},
  \end{align}
  and hence the result follows.
\end{proof}

\section{Splitting and weighted \texorpdfstring{$H^2$}{H2} stability}\label{sec:H2_2}
In the following we show that the solution operator which map $\omega^{in}\mapsto F^{(1)}$
is not only bounded as an operator from $H^{1}$ to $L^{2}$ but also from $H^{2}$
to $H^{1}$.
In contrast, as studied in Section \ref{sec:splitting_beta}, $\omega^{in}\mapsto
\beta$ does not exhibit higher stability, but rather grows unbounded in
$L^{\infty}$ and $H^{1/2+}$ as time tends to infinity. However, we show that
stability holds in weighted spaces which still allow to establish the optimal
decay rates in the inviscid damping estimates.

We begin our study of higher regularity of $F^{(1)}$ starting from \eqref{eq:DyF1}-\eqref{eq:DyF2} and applying the derivative operator $\dg$. 
We obtain a largely
similar equation, where we again have to change our elliptic operator and
account for changed boundary data. For convenience, we define
\begin{align}\label{eq:Rdef}
R= F^{(1)} +\left[\de_{y} \left(\frac{U''}{U'}\dc\right)- \frac{U''}{U'}\dc^{(1)}- 2 (\de_{y}-iktU') \frac{U''}{U'}\dc\right]\left(\de_y-iktU'\right)\Psi,
\end{align}
the right-hand side of \eqref{eq:DyF2}

\begin{lemma}\label{lem:F2}
Let $F^{(1)}$ be as in Lemma \ref{lem:splitting} and assume that the
  assumptions of Theorem \ref{thm:L2} are satisfied.
  Then $\dg F^{(1)}$ is the unique solution of the following equation
  \begin{align}
\de_{t} \dg F^{(1)} &=  ik B \ttPhi+ik (\dg B)\tPsi + ik\dg[(\dg B) \Psi ]\notag\\
&\quad+  \dc(\a)\l  F^{(1)},\frac{ \e^{ikt(U(y)-U(\a))}}{tU'd} h_\a  \r \dg( B\e^{{ikt(U(y)-U(\a))}}\tilde{h}_{\a}) \notag\\
&\quad-  \dc(\b)\l  F^{(1)},\frac{ \e^{ikt(U(y)-U(\b))}}{tU'd} h_\b  \r    \dg(B\e^{{ikt(U(y)-U(\b))}}\tilde{h}_{\b})\notag\\
&\quad+   \dc(\a)\l  \de_y\left(\frac{h_\a}{tU'}\right)F,\e^{ikt(U(y)-U(\a))}  \r  \dg(B\e^{{ikt(U(y)-U(\a))}}\tilde{h}_{\a})\notag \\
&\quad-  \dc(\b)\l  \de_y\left(\frac{h_\b}{tU'}\right)F,\e^{ikt(U(y)-U(\b))}  \r    \dg(B  \e^{{ikt(U(y)-U(\b))}}\tilde{h}_{\b})\notag \\
&\quad-  \dc(\b)\l  \de_y\left(\frac{h_\b}{tU'}\right)F,\e^{ikt(U(y)-U(\b))}  \r    \dg(B  \e^{{ikt(U(y)-U(\b))}}\tilde{h}_{\b})\notag \\
&\quad+ ik B(\dg\tPsi|_{y=\a}) \e^{{ikt(U(y)-U(\a))}}\tilde{\tilde{h}}_{\a} +ik B(\dg\tPsi|_{y=b})    \e^{{ikt(U(y)-U(\b))}}\tilde{\tilde{h}}_{\b}
\end{align}
where
\begin{align}
&\left(-k^{2}+(\de_{y}-iktU')^{2}-4 (\de_{y}-iktU') \frac{\dc^{(1)}}{\dc}-\left(\frac{\dc^{(1)}}{\dc}\right)^2-2\frac{\dc^{(2)}}{d}\right) \ttPhi =\dg R \notag\\
&\qquad- 2 (\de_{y}-iktU') \frac{U''}{U'}\dc\left(\de_y-iktU'\right)\tPsi  
+\left(\frac{U'''}{U'}-2\left(\frac{U''}{U'}\right)^2\right)\dc\left(\de_y-iktU'\right)\tPsi\notag\\
&\qquad+2(\de_{y}-iktU')\dc \left(\frac{\dc^{(1)}}{\dc}\right)' \tPsi -2\dc^{(1)}\left(\frac{\dc^{(1)}}{\dc}\right)'\tPsi+2 \frac{U''}{U'}\dc\left(\de_y-iktU'\right)\frac{\dc^{(1)}}{\dc}\tPsi+\dc \left(\frac{\dc^{(2)}}{\dc}\right)'\tPsi
\end{align} 
with boundary conditions $\ttPhi|_{y=a,b} =0$. Above, for $\bullet=\a,\b$, we denoted by $\tilde{\tilde{h}}_\bullet$ the unique solution to
\begin{align}\label{eq:hheqn}
\left(-k^{2}+\de_{yy}-4 \de_{y} \frac{\dc^{(1)}}{\dc}-\left(\frac{\dc^{(1)}}{\dc}\right)^2-2\frac{\dc^{(2)}}{d}\right)\tilde{\tilde{h}}_\bullet=0
\end{align}
with boundary conditions
\begin{align}
\tilde{\tilde{h}}_{\a}(\a)=\tilde{\tilde{h}}_{\b}(\b)=1, \qquad \tilde{\tilde{h}}_{\a}(\b)=\tilde{\tilde{h}}_{\b}(\a)=0.\label{eq:boundhh}
\end{align} 
\end{lemma}

\begin{proof}
We apply $\dg$ to \eqref{eq:DyF1} and obtain
  \begin{align}\label{eq:montenegro}
\de_{t} \dg F^{(1)} &=  ik B \dg\tPsi+ik (\dg B)\tPsi + \dg[(\dg B)ik \Psi ]\notag\\
&\quad+  \dc(\a)\l  F^{(1)},\frac{ \e^{ikt(U(y)-U(\a))}}{tU'd} h_\a  \r \dg( B\e^{{ikt(U(y)-U(\a))}}\tilde{h}_{\a}) \notag\\
&\quad-  \dc(\b)\l  F^{(1)},\frac{ \e^{ikt(U(y)-U(\b))}}{tU'd} h_\b  \r    \dg(B\e^{{ikt(U(y)-U(\b))}}\tilde{h}_{\b})\notag\\
&\quad+   \dc(\a)\l  \de_y\left(\frac{h_\a}{tU'}\right)F,\e^{ikt(U(y)-U(\a))}  \r  \dg(B\e^{{ikt(U(y)-U(\a))}}\tilde{h}_{\a})\notag \\
&\quad-  \dc(\b)\l  \de_y\left(\frac{h_\b}{tU'}\right)F,\e^{ikt(U(y)-U(\b))}  \r    \dg(B  \e^{{ikt(U(y)-U(\b))}}\tilde{h}_{\b}),
\end{align}
We then split $\dg \tPsi$ into a non-homogenous solution $\ttPhi$ with homogeneous Dirichlet boundary conditions and
the linear combination of homogeneous solutions to \eqref{eq:hheqn}, namely
\begin{align}\label{eq:hhomoline}
\dg \tPsi=\ttPhi+ (\dg\tPsi|_{y=\a}) \e^{{ikt(U(y)-U(\a))}}\tilde{\tilde{h}}_{\a} +(\dg\tPsi|_{y=b})    \e^{{ikt(U(y)-U(\b))}}\tilde{\tilde{h}}_{\b}.
\end{align}
It remains to establish the equation satisfied by $\Phi^{(2)}$ and $\tilde{\tilde{h}}_{\bullet}$. Applying $\dg$ to equation \eqref{eq:DyF2}, we obtain
\begin{align}\label{eq:braulio}
&\left(-k^{2}+(\de_{y}-iktU')^{2}-2 (\de_{y}-iktU') \frac{\dc^{(1)}}{\dc}-\frac{\dc^{(2)}}{d}\right) \dg\tPsi =\dg R \notag\\
&\qquad\qquad+[ (\de_{y}-iktU')^{2},\dg]\tPsi-2 [(\de_{y}-iktU') \frac{\dc^{(1)}}{\dc},\dg]\tPsi -[\frac{\dc^{(2)}}{d},\dg]\tPsi,
\end{align}   
where $R$ is defined in \eqref{eq:Rdef}.
The first commutator above has been computed in \eqref{eq:comm1} (see also \eqref{eq:comm2}, \eqref{eq:comm3} and \eqref{eq:comm4}) as
\begin{align}\label{eq:commbraulio}
[(\de_{y}-iktU')^{2}, \dg] \Psi
    &=- 2 (\de_{y}-iktU') \frac{U''}{U'}\dc\left(\de_y-iktU'\right)\tPsi+ 2 (\de_{y}-iktU') \frac{\dc^{(1)}}{\dc} \dg \tPsi\notag\\
    &\quad+  \left(\frac{U''}{U'}\dc\right)'\left(\de_y-iktU'\right)\tPsi+\frac{\dc^{(2)}}{d}\dg\tPsi- \frac{U''}{U'}\dc^{(1)}\left(\de_y-iktU'\right)\tPsi. 
\end{align}
Moreover, since $[AB,C]=A[B,C]+[A,C]B$, we take advantage of \eqref{eq:comm0} to obtain that
\begin{align}
[(\de_{y}-iktU') \frac{\dc^{(1)}}{\dc},\dg]\tPsi
&=(\de_{y}-iktU') [\frac{\dc^{(1)}}{\dc},\dg]+[(\de_{y}-iktU') ,\dg]\frac{\dc^{(1)}}{\dc}\notag\\
&=-(\de_{y}-iktU')\dc \left(\frac{\dc^{(1)}}{\dc}\right)' +\left(\dc^{(1)}\de_y- \frac{U''}{U'}\dc\left(\de_y-iktU'\right)\right)\frac{\dc^{(1)}}{\dc}\notag\\
&=-(\de_{y}-iktU')\dc \left(\frac{\dc^{(1)}}{\dc}\right)' +\dc^{(1)}\left(\frac{\dc^{(1)}}{\dc}\right)'+\left(\frac{\dc^{(1)}}{\dc}\right)^2\dg- \frac{U''}{U'}\dc\left(\de_y-iktU'\right)\frac{\dc^{(1)}}{\dc}
\end{align}
and
\begin{align}
[\frac{\dc^{(2)}}{d},\dg]=- \dc \left(\frac{\dc^{(2)}}{\dc}\right)'.
\end{align}
As in the proof of Lemma \ref{lem:H1equation}, we consider terms involving $\dg$ as part of the modified elliptic
operator and put every other term into the inhomogeneity. Hence, from \eqref{eq:braulio} we find that
\begin{align}\label{eq:braulio1}
&\left(-k^{2}+(\de_{y}-iktU')^{2}-4 (\de_{y}-iktU') \frac{\dc^{(1)}}{\dc}-\left(\frac{\dc^{(1)}}{\dc}\right)^2-2\frac{\dc^{(2)}}{d}\right) \dg\tPsi =\dg R \notag\\
&\qquad- 2 (\de_{y}-iktU') \frac{U''}{U'}\dc\left(\de_y-iktU'\right)\tPsi  
+\left(\frac{U'''}{U'}-2\left(\frac{U''}{U'}\right)^2\right)\dc\left(\de_y-iktU'\right)\tPsi\notag\\
&\qquad+2(\de_{y}-iktU')\dc \left(\frac{\dc^{(1)}}{\dc}\right)' \tPsi -2\dc^{(1)}\left(\frac{\dc^{(1)}}{\dc}\right)'\tPsi+2 \frac{U''}{U'}\dc\left(\de_y-iktU'\right)\frac{\dc^{(1)}}{\dc}\tPsi+\dc \left(\frac{\dc^{(2)}}{\dc}\right)'\tPsi
\end{align} 
The proof is concluded by plugging the above linear combination \eqref{eq:hhomoline} into \eqref{eq:montenegro}.
\end{proof}

Following a similar strategy as in the previous section, we show that for a
mildly degenerate flow the operator mapping $F$ to $\ttPhi$ can be estimated
by our standard elliptic operator, similarly to Lemma \ref{lem:H1elliptic}.

\begin{lemma}\label{lem:H2elliptic}
In the setting of Lemma \ref{lem:F2}, suppose that \ref{P} is satisfied.
Then for any interval $J$ and any given function $g \in L^{2}(J)$, the unique solution $\psi$ of  
\begin{align}
    \left(-k^{2}+(\de_{y}-iktU')^{2}-4 (\de_{y}-iktU') \frac{\dc^{(1)}}{\dc}-\left(\frac{\dc^{(1)}}{\dc}\right)^2-2\frac{\dc^{(2)}}{d}\right)\psi &=g \\
    \psi|_{\de J}&=0
  \end{align}
  and the unique solution $\psi^{(2)}$ of
  \begin{align}
    (k^2+(\de_{y}-iktU')^{2})\psi^{(2)}&= g, \\
    \psi^{(1)}|_{\de J}&=0
  \end{align}
  satisfy 
    \begin{align}
    \int_{I} |U'(y)| |\nabla_{k,t}\psi(y)|^{2}\dd y \leq C \int_{I} |U'(y)| |\nabla_{k,t}\psi^{(2)}(y)|^{2}\dd y. 
  \end{align}
\end{lemma}
\begin{proof}
Repeating the steps to obtain \eqref{eq:loaef} in the proof of Lemma \ref{lem:H1elliptic}, we obtain
\begin{align}
\eps \int_J |U'(y)||\nabla_{k,t} \psi(y)|^2\dd y \leq\Re\l \nabla_{k,t} \psi,\nabla_{k,t}(|U'|\psi)\r.
\end{align}
  Conversely, we may estimate
  \begin{align}
    \langle -|U'|\psi, g \rangle = \langle -|U'|\psi, (k^{2}+(\de_{y}-iktU')^{2})\psi^{(2)} \rangle \leq C\int_J |U'(y)||\nabla_{k,t} \psi(y)||\nabla_{k,t} \psi^{(2)}(y)|\dd y.
  \end{align}
  The result hence follows by H\"older's and Young's inequalities.
\end{proof}

\begin{lemma}
  \label{lem:boundaryCorrection2}
  In the setting of Lemma \ref{lem:F2} let $g \in L^{2}$ be any given function
  and let $(\chi_{j}g)_{n}$ denote the Fourier basis expansion of $\chi_{j}g$ on
  the interval $I_{j}=(a_{j},b_{j})$ with respect to
  $z=\frac{U(y)-U(a_{j})}{\min_{I_{j}} |U'|} \in (0,1)$.
  Let further $\tilde{\tilde{h}}$ be as in Lemma \ref{lem:F2} and $l \in
  W^{1,\infty}_{loc}$. Then it holds that
  \begin{align}
    | \langle g, l \e^{ikt U} \tilde{\tilde{h}} \rangle | 
    &\leq \sum_{j,n} \frac{|(\chi_{j}g)_{n}|}{|k|+|n-kt\min U'|} \|\chi_{j}l \tilde{\tilde{h}}\|_{H^{1}} \\
    &\leq C_{\delta} \|l \frac{\tilde{\tilde{h}}}{U'}\|_{H^{1}} \sum_{j,n} \frac{|(\chi_{j}g)_{n}|^{2} \min U'}{(|k|+|n-kt\min U'|)^{1-\delta}}.
  \end{align}
\end{lemma}

\begin{proof}
  This result follows as in Lemma \ref{lem:boundaryCorrection}. We first expand
  the scalar product in terms of $\chi_{j}$ and then expand in a Fourier basis
  on each interval $I_{j}$.
  More precisely, as in Lemma \ref{lem:boundaryCorrection}, we expand with
  respect to
  \begin{align}
    e_{n}(y)= \frac{1}{\|\frac{U'}{c}\|_{L^{2}(I_{j},dy)}} \exp\left( in \frac{U(y)-U(a_{j})}{c} l_{j}\right).
  \end{align}
  In order to simplify notation, we again rescale $n$ by $l_{j}$ and note that
  the normalizing factor provides a correct transformation.
  We note that
  \begin{align}
    \langle \e^{inz}, \e^{iktU}\chi_{j}l \tilde{\tilde{h}} \rangle
  \end{align}
  can be expressed using $\e^{i(n-kt\min U')z}$ by our choice of coordinate $z$.
  A first estimate hence follows by integration by parts.
  However, this is estimate is suboptimal if $n-kt\min U'$ is small.
  In order to improve this estimate, we further use the structure of
  $\tilde{\tilde{h}}$, which decays exponentially like $\e^{-|ky|}$ just like $h$.
\end{proof}

\begin{lemma}
  \label{lem:coefficientsF2}
  Let $U, h_{a},h_{b}$ be as in Lemma \ref{lem:boundaryCorrection2} and
  let $l \in W^{1,\infty}_{loc}$.
Then for any $g \in H^{1}$ and $\bullet \in \{a,b\}$ it holds that
\begin{align}
  |\langle g, l \e^{ikt(U-U(\bullet))} h_{\bullet} \rangle| \leq \left|\frac{gl}{ktU'}\bigg|_{y=\bullet}\right|
+ \sum_{j,n} \frac{|(\chi_{j}g)_{n}| + (|\chi_{j}\dg g)_{m}|}{|k|+|n-kt\min U'|} \|\frac{\chi_{j}l h_{\bullet}}{U'}\|_{H^{2}}.
\end{align}
\end{lemma}

\begin{proof}
  We integrate $\e^{ikt (U(y)-U(\bullet))} = \frac{1}{iktU'}\de_{y}\e^{ikt (U(y)-U(\bullet))}$
  by parts and use that $h_{\bullet}$ is $1$ on the boundary point $y=\bullet$ and vanishes on
  the other.
  The result hence follows as in Lemma \ref{lem:boundaryCorrection2} by
  expressing the resulting inner product as
  \begin{align}
    \langle \dg g, \e^{ikt(U-U(\bullet))} \frac{h_{\bullet}l}{\dc ktU'} \rangle + \langle g, \e^{ikt(U-U(j))} \de_{y} \frac{h_{\bullet}l}{ktU'} \rangle. 
  \end{align}
\end{proof}
We remark that for $g=F^{(1)}$ the boundary evaluation satisfies 
\begin{align}
  \de_t F^{(1)}|_{y=\bullet} =  (\langle F^{(1)}, \frac{\e^{{ikt(U(y)-U(\bullet))}}}{iktU'} \frac{h_{\bullet}}{d} \rangle +\langle F, \e^{{ikt(U(y)-U(\bullet))}}\de_{y}(\frac{h_{\bullet}}{iktU'})  \rangle )ikB |_{y=\bullet} = \mathcal{O}(t^{-1})\|\omega^{in}\|_{H^{1}}
\end{align}
by the preceding results. Hence, we obtain a logarithmic growth bound.

With these preparations, we can establish stability of $\dg F^{(1)}$.
\begin{proposition}
  \label{prop:DF1}
  Let $F^{(1)}$ be as in Lemma \ref{lem:H1equation} and suppose that the
  assumptions of Lemma \ref{lem:boundaryCorrection} hold.
  Let further $A_{1}(t)$ be as in Lemma \ref{lem:modifiedMultiplier}.
  Then there exists constant $C_{1}, C_{2}$ such that for $t \gg 1$ 
  \begin{align}
\ddt  \left(  \langle \dg F^{(1)}, A_{1}(t) \dg F^{(1)} \rangle + C_{1}\langle F^{(1)}, A_{1}(t) F^{(1)} \rangle + C_{2}  \langle F, A_{1}(t) F \rangle \right) \leq Ct^{-1-\delta}
  \end{align}
\end{proposition}

\begin{proof}
Following a similar strategy as in the proof of Proposition \ref{prop:F1}, we
separately estimate the contributions in
\begin{align}
  \langle \de_t \dg F^{(1)}, A_{1}(t) \dg F^{(1)} \rangle
\end{align}
involving $\ttPhi, \tPhi$ and $\Phi$ and the contributions due to
homogeneous corrections $h, \tilde{h}$ and $\tilde{\tilde{h}}$.
For the latter terms, we rely on Lemmas
\ref{lem:coefficientsF2}, \ref{lem:boundaryCorrection2}
and \ref{lem:boundaryCorrection}.
Using the properties of $\dot A_{1}(t)$ established in Lemma
\ref{lem:modifiedMultiplier} and Young's inequality, these contributions can be
absorbed in
\begin{align}
   \langle \dg F^{(1)}, \dot A_{1}(t) \dg F^{(1)} \rangle + C_{1}\langle F^{(1)}, \dot A_{1}(t) F^{(1)} \rangle + C_{2}  \langle F, \dot A_{1}(t) F \rangle,
\end{align}
provided $C_{1}, C_{2}$ are sufficiently big and $t \gg 1$.

In order to control the contributions due to $\Phi^{(2)}, \Phi^{(1)}$ and
$\Phi$,
we first note that contributions due to $\Phi^{(1)}$ and $\Phi$ can be estimated
as in Section \ref{sec:higher}, where we use control of $\dg B$ and $\dg ^{2}B$ and
\eqref{eq:hhomoline}.
It thus only remains to control $\Phi^{(2)}$.
Here, Lemma \ref{lem:H2elliptic} allows us to reduce to our previous elliptic
operator.
In the notation of Lemma \ref{lem:F2}, we integrate $(\de_{y}-iktU')$ terms on
the right-hand-side by parts and use Young's inequality to control in terms of
$\|\Phi^{(1)}\|_{H^{1}_{t}(U')}, \|\Phi^{(1)}\|_{H^{1}_{t}(U')}$ and the stream
function corresponding to a right-hand-side $\dg F^{(1)}$.
Using again the reduction of Lemma \ref{lem:H2elliptic} and the estimates
established in the preceeding sections, these estimates are then localized and
our weights $A_{1}(t)$ were constructed in just such a way that the localized
streamfunction contributions can be absorbed.
\end{proof}

It hence remains to control $\dg\beta_{\bullet}$. Here, as in \cite{Zill5} and in
\cite{Zill6}, we split off a boundary layer $\nu$ that asymptotically diverges in
unweighted $L^{2}$ and more well-behaved part $\gamma$.
Compared to previous works, one additional challenge here is that the degeneracy
of this layer also depend on $\frac{1}{U'}$ and thus an understanding of
stability in weighted spaces is necessary.

\subsection{Splitting of  \texorpdfstring{$\dg \beta_\bullet$}{D\beta}}
\label{sec:splitting_beta}
When computing the equation satisfied by $\dg\beta$, we note that many terms are
quite similar to the ones appearing in the equation of $\dg F^{(1)}$ or $\beta$
itself.
However, as it is clear from the last term of the right-hand side of \eqref{eq:beta}, there is also is a contribution by
\begin{align}
 \frac{\dc\omega^{in}}{tU'}\bigg|_{y=\bullet} B\dg(\e^{{ikt(U(y)-U(\bullet))}})\tilde{h}_{\bullet}=
  \frac{\dc\omega^{in}}{U'}\bigg|_{y=\bullet} ikBU'\e^{{ikt(U(y)-U(\bullet))}}\tilde{h}_{\bullet},
\end{align}
which does not exhibit sufficient decay and oscillation in time to be an
integrable contribution.
Hence, we split this inhomogeneity off as a separate boundary layer.
Unlike in \cite{Zill6} the solution operator of the homogeneous solution operator
not only involves several contributions due to different homogeneous corrections
$h,\tilde{h}$ and $\tilde{\tilde{h}}$, but also a modified elliptic operator.
Hence, in a Duhamel approach we have to take care to control these various
corrections and further have establish conditional higher regularity results
which are used as estimates inside Duhamel's formula.

We begin by introducing a splitting. Subsequently, we develop bounds on the
solution operator for the linear propagator and estimates on a Duhamel-type
integral in weighted Sobolev spaces. Similarly to what we did in \eqref{eq:homo1}-\eqref{eq:homo4},
we introduce the auxiliary functions $\tilde{h}^\star_\bullet$, with $\bullet=\a,\b$, solutions to
\begin{align}
&\left(-k^{2}+\de_{yy} +2\frac{\dc^{(1)}}{\dc}\de_{y}  -\frac{\dc^{(2)}}{\dc} \right)\tilde{h}^\star_\bullet=0,\label{eq:hstaromo1} 
\end{align}
with boundary conditions
\begin{align}
\tilde{h}^\star_{\a}(\a)=\tilde{h}^\star_{\b}(\b)=1, \qquad \tilde{h}^\star_{\a}(\b)=\tilde{h}^\star_{\b}(\a)=0. \label{eq:boundhstar}
\end{align} 
In particular,
\begin{align}
\left(-k^{2}+(\de_{y}-iktU')^{2} +2\frac{\dc^{(1)}}{\dc} (\de_{y}-iktU') -\frac{\dc^{(2)}}{\dc} \right)\e^{iktU}\tilde{h}^\star_\bullet=0.\label{eq:hstaromo3} 
 \end{align}
We then have the following formulation for $\dg \beta_\bullet$.
\begin{lemma}\label{lem:Db_splitting}
  Let $\beta_{\bullet}$ be as in Lemma \ref{lem:splitting}, then we may
  decompose $\dg \beta_{\bullet}= \gamma_{\bullet}+\nu_{\bullet}$. These are
  well-behaved part and a more singular boundary layer, respectively, and solve
\begin{align}
&\de_{t}  \gamma_\bullet = ikB \ttPsi_\bullet +ik (\dg B) \tPsi_\bullet \notag\\
&\quad+ B \left(\frac{ \beta_\a}{ \dc tU'}\bigg|_{y=\a} 
+\frac{1}{\dc(\a)}\l \frac{\tilde{h}^\star_\a}{\dc tU'}  \e^{ikt(U(y)-U(\a))}, \gamma_\a \r
+\frac{1}{\dc(\a)}\l \de_y\left(\frac{\tilde{h}^\star_\a}{\dc tU'}\right)  \e^{ikt(U(y)-U(\a))}, \beta_\a \r\right) \e^{{ikt(U(y)-U(\a))}}\tilde{\tilde{h}}_{\a} \notag\\
&\quad+B \left(\frac{ \beta_\b}{ \dc tU'}\bigg|_{y=\b} 
+\frac{1}{\dc(\b)}\l \frac{\tilde{h}^\star_\b}{\dc tU'}  \e^{ikt(U(y)-U(\b))},\gamma_\b \r
+\frac{1}{\dc(\b)}\l \de_y\left(\frac{\tilde{h}^\star_\b}{\dc tU'}\right)  \e^{ikt(U(y)-U(\b))}, \beta_\b \r\right)    \e^{{ikt(U(y)-U(\b))}}\tilde{\tilde{h}}_{\b}\notag\\
&\quad- \dc(\a)\left(\frac{ \beta_\bullet}{ik d(tU')^2}\bigg|_{y=\a}   +\l \gamma_\bullet ,\frac{h_\a}{ik (\dc tU')^2}  \e^{ikt(U(y)-U(\a))} \r
+\l  \beta_\bullet,\de_y\left(\frac{h_\a}{ik \dc (tU')^2}  \right)\e^{ikt(U(y)-U(\a))}\r \right) \notag\\ 
&\qquad\qquad\times\dg (B\e^{{ikt(U(y)-U(\a))}}\tilde{h}_{\a})\notag \\
&\quad-  \dc(\b)\left(\frac{ \beta_\bullet}{ik d(tU')^2}\bigg|_{y=\b}   -\l \gamma_\bullet ,\frac{h_\b}{ik (\dc tU')^2}  \e^{ikt(U(y)-U(\b))} \r
-\l  \beta_\bullet,\de_y\left(\frac{h_\b}{ik \dc (tU')^2}  \right)\e^{ikt(U(y)-U(\b))}\r  \right)  \notag\\
&\qquad\qquad\times\dg (B\e^{{ikt(U(y)-U(\b))}}\tilde{h}_{\b})\notag\\
&\quad +\frac{\dc\omega^{in}}{tU'}\bigg|_{y=\bullet} \dg (B\tilde{h}_{\bullet})\e^{{ikt(U(y)-U(\bullet))}},
\end{align}
with initial condition $\gamma_{\bullet}|_{t=0}=0$, where
\begin{align}
&\left(-k^{2}+(\de_{y}-iktU')^{2}-4 (\de_{y}-iktU') \frac{\dc^{(1)}}{\dc}-\left(\frac{\dc^{(1)}}{\dc}\right)^2-2\frac{\dc^{(2)}}{d}\right)\ttPsi_\bullet= \gamma_\bullet \notag\\
&\qquad- 2 (\de_{y}-iktU') \frac{U''}{U'}\dc\left(\de_y-iktU'\right)\tPsi_\bullet  
+\left(\frac{U'''}{U'}-2\left(\frac{U''}{U'}\right)^2\right)\dc\left(\de_y-iktU'\right)\tPsi_\bullet\notag\\
&\qquad+2(\de_{y}-iktU')\dc \left(\frac{\dc^{(1)}}{\dc}\right)' \tPsi_\bullet -2\dc^{(1)}\left(\frac{\dc^{(1)}}{\dc}\right)'\tPsi_\bullet+2 \frac{U''}{U'}\dc\left(\de_y-iktU'\right)\frac{\dc^{(1)}}{\dc}\tPsi_\bullet+\dc \left(\frac{\dc^{(2)}}{\dc}\right)'\tPsi_\bullet,
\end{align}
with boundary conditions $\ttPsi_\bullet|_{y=\a,\b}=0$.
The boundary layer $\nu_\bullet$ solves
\begin{align}
\de_{t}  \nu_\bullet &=  ikB \ttpsi + \frac{B }{\dc(\a)}\l \frac{\tilde{h}^\star_\a}{\dc tU'}  \e^{ikt(U(y)-U(\a))},\nu_\a \r
 \e^{{ikt(U(y)-U(\a))}}\tilde{\tilde{h}}_{\a}\notag \\
&\quad+ \frac{B }{\dc(\b)}\l \frac{\tilde{h}^\star_\b}{\dc tU'}  \e^{ikt(U(y)-U(\b))},\nu_\b \r    \e^{{ikt(U(y)-U(\b))}}\tilde{\tilde{h}}_{\b}\notag\\
&\quad- \dc(\a) \l \nu_\bullet,\frac{h_\a}{ik (\dc tU')^2}  \e^{ikt(U(y)-U(\a))} \r   \dg (B\e^{{ikt(U(y)-U(\a))}}\tilde{h}_{\a})\notag \\
&\quad+  \dc(\b) \l \nu_\bullet,\frac{h_\b}{ik (\dc tU')^2}  \e^{ikt(U(y)-U(\b))} \r    \dg (B\e^{{ikt(U(y)-U(\b))}}\tilde{h}_{\b})\notag\\
&\quad+ ik\frac{\dc\omega^{in}}{U'}\bigg|_{y=\bullet} BU'\e^{{ikt(U(y)-U(\bullet))}}\tilde{h}_{\bullet},
\end{align}
with initial condition $\nu_{\bullet}|_{t=0}=0$, where
\begin{align}
&\left(-k^{2}+(\de_{y}-iktU')^{2}-4 (\de_{y}-iktU') \frac{\dc^{(1)}}{\dc}-\left(\frac{\dc^{(1)}}{\dc}\right)^2-2\frac{\dc^{(2)}}{d}\right)\ttpsi_\bullet= \nu_\bullet.
\end{align}
with boundary conditions $\ttpsi_\bullet|_{y=\a,\b}=0$.
\end{lemma}

\begin{proof}
For $\bullet=a,b$, from \eqref{eq:beta} we directly compute
\begin{align}\label{eq:Dbeta}
\de_{t}  \dg\beta_\bullet &= ikB (\dg \tPsi_\bullet)+ik (\dg B) \tPsi_\bullet+ \dc(\a)\l  \beta_\bullet,\frac{ \e^{ikt(U(y)-U(\a))}}{tU'\dc} h_\a  \r  \dg (B\e^{{ikt(U(y)-U(\a))}}\tilde{h}_{\a})\notag \\
&\quad-  \dc(\b)\l  \beta_\bullet,\frac{ \e^{ikt(U(y)-U(\b))}}{tU'\dc} h_\b  \r    \dg (B\e^{{ikt(U(y)-U(\b))}}\tilde{h}_{\b})
+    \frac{\dc\omega^{in}}{tU'}\bigg|_{y=\bullet} \dg (B\e^{{ikt(U(y)-U(\bullet))}}\tilde{h}_{\bullet}),
\end{align}
The treatment of $\dg\tPsi_\bullet$ is exactly the same as in the proof of Lemma \ref{lem:F2}, with the splitting according to \eqref{eq:hhomoline}.
In order to conclude our proof, we note the following integration by parts results for variable coefficients. Firstly,
using the boundary conditions \eqref{eq:boundh1}-\eqref{eq:boundh2}, we have 
\begin{align}
\l  \beta_\bullet,\frac{ \e^{ikt(U(y)-U(\a))}}{tU'd} h_\a  \r
&=\l  \beta_\bullet,\frac{h_\a}{ik d(tU')^2}  \de_y\e^{ikt(U(y)-U(\a))} \r\notag\\
&=-\frac{ \beta_\bullet}{ik d(tU')^2}\bigg|_{y=\a}   -\l \dg \beta_\bullet,\frac{h_\a}{ik (\dc tU')^2}  \e^{ikt(U(y)-U(\a))} \r\notag\\
&\quad-\l  \beta_\bullet,\de_y\left(\frac{h_\a}{ik \dc (tU')^2}  \right)\e^{ikt(U(y)-U(\a))}\r ,
\end{align}
and, analogously,
\begin{align}
\l  \beta_\bullet,\frac{ \e^{ikt(U(y)-U(\b))}}{tU'd} h_\b  \r&=\frac{ \beta_\bullet}{ik d(tU')^2}\bigg|_{y=\b}   -\l \dg \beta_\bullet,\frac{h_\b}{ik (\dc tU')^2}  \e^{ikt(U(y)-U(\b))} \r\notag\\
&\quad-\l  \beta_\bullet,\de_y\left(\frac{h_\b}{ik \dc (tU')^2}  \right)\e^{ikt(U(y)-U(\b))}\r.
\end{align}
Moreover, as in \eqref{eq:boundaryterms}, we use \eqref{eq:betapsi}, \eqref{eq:hstaromo3}, 
the boundary conditions \eqref{eq:boundhstar} and $\tPsi_\bullet |_{y=\a,\b}=0$ to obtain on the one hand that
\begin{align}
\l \e^{{ikt(U(y)-U(a))}} \tilde{h}^\star_{\a}, \beta_{\a}  \r
&=\l\e^{{ikt(U(y)-U(a))}} \tilde{h}^\star_{\a}, \left(-k^{2}+(\de_{y}-iktU')^{2}-2 (\de_{y}-iktU') \frac{\dc^{(1)}}{\dc}-\frac{\dc^{(2)}}{d}\right) \tPsi_\a  \r\notag\\
&=-\de_y \tPsi |_{y=\a}
\end{align}
and, on the other hand, that
\begin{align}
\l \e^{{ikt(U(y)-U(a))}} \tilde{h}^\star_{\a}, \beta_{\a}  \r&=-\frac{ \beta_\a}{ik tU'}\bigg|_{y=\a} 
-\l \frac{\tilde{h}^\star_\a}{ik\dc tU'}  \e^{ikt(U(y)-U(\a))},\dg \beta_\a \r\notag\\
&\quad-\l \de_y\left(\frac{\tilde{h}^\star_\a}{ik\dc tU'}\right)  \e^{ikt(U(y)-U(\a))}, \beta_\a \r
\end{align}
Hence,
\begin{align}
\dg\tPsi|_{y=\a}&=\frac{ \beta_\a}{ik \dc tU'}\bigg|_{y=\a} 
+\frac{1}{\dc(\a)}\l \frac{\tilde{h}^\star_\a}{ik\dc tU'}  \e^{ikt(U(y)-U(\a))},\dg \beta_\a \r\notag\\
&\quad+\frac{1}{\dc(\a)}\l \de_y\left(\frac{\tilde{h}^\star_\a}{ik\dc tU'}\right)  \e^{ikt(U(y)-U(\a))}, \beta_\a \r
\end{align}
Analogously,
\begin{align}
\dg\tPsi|_{y=\b}&=\frac{ \beta_\b}{ik \dc tU'}\bigg|_{y=\b} 
+\frac{1}{\dc(\b)}\l \frac{\tilde{h}^\star_\b}{ik\dc tU'}  \e^{ikt(U(y)-U(\b))},\dg \beta_\b \r\notag\\
&\quad+\frac{1}{\dc(\b)}\l \de_y\left(\frac{\tilde{h}^\star_\b}{ik\dc tU'}\right)  \e^{ikt(U(y)-U(\b))}, \beta_\b \r
\end{align}
By restricting to the boundary the evolution equation of $\beta_\bullet$ in \eqref{eq:beta} and using that $\tPsi_\bullet$ vanishes, 
we obtain
\begin{align}
\de_{t}  \beta_\bullet|_{y=a,b}  &=B \dc(\a)\l  \beta_\bullet,\frac{ \e^{ikt(U(y)-U(\a))}}{tU'd} h_\a  \r  \e^{{ikt(U(y)-U(\a))}}\tilde{h}_{\a}|_{y=a,b}\notag \\
&\quad-  B\dc(\b)\l  \beta_\bullet,\frac{ \e^{ikt(U(y)-U(\b))}}{tU'd} h_\b  \r    \e^{{ikt(U(y)-U(\b))}}\tilde{h}_{\b}|_{y=a,b}
+    \frac{\dc\omega^{in}}{tU'}\bigg|_{y=\bullet} B\e^{{ikt(U(y)-U(\bullet))}}\tilde{h}_{\bullet}|_{y=a,b},
\end{align}
As the homogeneous solutions $\e^{{ikt(U(y)-U(\bullet))}} \tilde{h}_{\bullet}$ are
chosen with boundary values $0$ and $1$ and $\|\beta_\bullet(t)\|_{L^{2}}$ is uniformly
bounded by Proposition \ref{prop:F1}, we may restrict the evolution equation of
$\beta_{\bullet}$ in Lemma \ref{lem:splitting} to the boundary and obtain that
\begin{align*}
  |\dt \beta_{\bullet}|_{\partial I}| \leq \frac{C}{t} |B d|_{\partial I}| (\|\beta_{\bullet}\|_{L^{2}} + |\omega_{in}|_{\partial I}|) \leq \frac{C}{t}. 
\end{align*}
Integrating in time (using Gronwall's lemma for small times), it follows that
\begin{align}
  |\beta_{\bullet}||_{y=a,b} \leq C \log(2+t) |\omega^{in}|_{y=\bullet}|
\end{align}
as $t$ tends to infinity.

Thus, the first term in the expansion of $\dg\tPsi|_{y=\bullet}$ satisfies
\begin{align}
 \left| \frac{ \beta_\a}{ik \dc tU'}\bigg|_{y=\bullet} \right|\leq C \frac{\log(2+t)}{t} |\omega^{in}|_{y=\bullet}|.
\end{align}
Since this term is square integrable, even after multiplication by $t^{\delta}$,
we consider it a given lower order contribution. 
Having identified the terms linear in $\dg\beta_\bullet$ above, we next split 
$\dg\beta_\bullet= \gamma_\bullet+\nu_\bullet$, with $\nu_\bullet$ having inhomogeneity 
\begin{align}
  \label{eq:1}
ik\frac{\dc\omega^{in}}{U'}\bigg|_{y=\bullet} BU'\e^{{ikt(U(y)-U(\bullet))}}\tilde{h}_{\bullet}
\end{align}  
and $\gamma_\bullet$ incorporating all other inhomogeneities, which establishes
the result.
\end{proof}

The just introduced splitting allows us to separately study the possible
growth due to the contribution \eqref{eq:1}. As this term does not possess time
decay of the absolute value and does not oscillate near the boundary, stability
in unweighted spaces does not hold for this term. Instead, we establish
stability in a weighted space using the simplified dependence compared to $\dg
\beta$ in a Duhamel's formula based approach. 

\subsection{Stability of  \texorpdfstring{$\gamma_\bullet$}{\gamma}}\label{sec:gamma}

In the preceding splitting, we have determined the contribution to \eqref{eq:1}
as the potentially most difficult to control. The following proposition
complements this understanding by showing that $\gamma_{\bullet}$, which evolves
with this source of instability removed, enjoys similar stability estimates as $F^{(2)}$.
\begin{proposition}
  \label{prop:gamma}
  Under the same assumptions as in Proposition \ref{prop:DF1} and with $A^{(1)}$
  as in Lemma \ref{lem:modifiedMultiplier}, there exists constants $C_{1},
  C_{2}, C_{3}>0$ such that for $t\gg 1$
  \begin{align}
    \de_t  \left(  \langle \gamma_{\bullet}, A_{1}(t) \gamma_{\bullet} \rangle + C_{1}\langle F^{(1)}, A_{1}(t) F^{(1)} \rangle + C_{2}  \langle F(t), A_{1}(t) F(t) \rangle  +C_{3}\langle \beta_{\bullet}, A_{1}(t) \beta_{\bullet} \rangle  \right) \leq Ct^{-1-\delta}
  \end{align}
\end{proposition}

\begin{proof}
  The bound for all but the first term have already been established in the
  preceding sections.
  In order to control
  \begin{align}
    \de_t  \langle \gamma_{\bullet}, A_{1}(t) \gamma_{\bullet} \rangle
  \end{align}
  we proceed as in the proof of Proposition \ref{prop:DF1}.
  That is, terms involving the stream functions $\ttPsi_{\bullet}$ and
  $\tPsi_{\bullet}$ can be estimated in term of the standard elliptic operators
  using Lemma \ref{lem:H2elliptic} and subsequently localized and absorbed in
  the decay of the weights.
  For the boundary corrections, we similarly use Lemmas
  \ref{lem:boundaryCorrection2} and \ref{lem:coefficientsF2} as well as the
  growth bound for boundary evaluations established there. 
\end{proof}

\subsection{Weighted stability of  \texorpdfstring{$\nu_\bullet$}{\nu}}
In order to establish stability, we intend to use that 
\begin{align}
ik \frac{\dc\omega^{in}}{U'}\bigg|_{y=\bullet}\int^{T}_{1}  BU'\e^{{ikt(U(y)-U(\bullet))}}\tilde{h}_{\bullet} \dd t
= \frac{\dc\omega^{in}}{U'}\bigg|_{y=\bullet}   BU'\frac{\e^{ikT(U(y)-U(\bullet))}-\e^{ik(U(y)-U(\bullet))}-1}{(U(y)-U(\bullet)}\tilde{h}_{\bullet}.
\end{align}
Since $\e^{ikT(U(y)-U(\bullet))}-\e^{ik(U(y)-U(\bullet))}$ vanishes for $y=\bullet$, the
singularity  is removed for $T<\infty$, but as $T \rightarrow \infty$ the
$L^{p}$ norms diverge.
However, if we weigh with $\min\{|U(y)-U(\bullet)|,1\}$, this problem is solved and we
obtain uniform bounds. The main difficulty in implementing this heuristic is that in Duhamel's formula
we further have to apply our evolution operator before integrating in time.

\begin{itemize}
\item In order to make use of the oscillation, we want a straight forward
  commutation relation with multiplication by $\e^{ik \tau U}$.
  However, we do not possess an explicit commutator for terms involving
  $\frac{\e^{iktU}}{ikt U'} h, \tilde{h}, \tilde{\tilde{h}}$ due to
  $\frac{1}{iktU'}$ not being given by conjugation.
  Hence, we split our linear propagator and also the solution to make use of the
  better commutator structure of the other terms. 
\item In order to control errors in the time integral we make use of uniform
  estimates on the stream function in terms of higher Sobolev norms. However, as
  also the elliptic operator is modified we obtain further modified operators
  when considering higher regularity.
\end{itemize}

\begin{lemma}
  \label{lem:evolution}
  Under the same assumptions as in Theorem \ref{thm:L2} and with the notation of
  Lemma \ref{lem:H1equation} let $S(t_{2},t_{1}): L^{2} \rightarrow L^{2}$ denote the solution operator of
  \begin{align}
    \de_t u &= ik B \psi, \\
\left(-k^{2}+(\de_{y}-iktU')^{2}-4(\de_{y}-iktU') \frac{\dc^{(1)}}{\dc}-\left(\frac{\dc^{(1)}}{\dc}\right)^2-2\frac{\dc^{(2)}}{d}\right)\psi  &= u \label{eq:psidefi} \\
  \psi |_{y=a,b}&=0
  \end{align}
  Then for any $u_{0} \in L^{2}$ or $H^1$, respectively, and all $t_{2}\geq t_{1}$ it holds that
  \begin{align}
    \|S(t_{2},t_{1})u_{0}\|_{L^{2}} &\leq C \|u_{0}\|_{L^{2}}, \\
    \|S(t_{2},t_{1})u_{0}\|_{H^{1}} &\leq C \|u_{0}\|_{H^{1}}.
  \end{align}
  Furthermore, for any $\tau \in \R$
  \begin{align}
    S(t_{2},t_{1})\e^{ik\tau U} u_{0} = \e^{ik\tau U} S(t_{2}-\tau,t_{1}-\tau) u_{0}.
  \end{align}
\end{lemma}

\begin{proof}
We divide the proof into four steps.

\medskip

\noindent \underline{$L^{2}$ stability}:
Let $A^{(1)}(t)=\sum_{j} \chi_{j} A^{(1)}_{j}(t)$ with $A_{j}^{(1)}(t)$ as in
 Lemma \ref{lem:modifiedMultiplier} and consider the energy functional
\begin{align}
  \label{eq:5}
  I(t):=\langle u(t), A^{(1)}(t) u(t) \rangle.
\end{align}
Following the same approach as in the proof of Proposition \ref{prop:DF1}, it follows that
\begin{align}
  \ddt I= \langle u, \dot A^{(1)}(t) u \rangle + 2 \Re \langle ikB \psi, A^{(1)}(t)u \rangle 
  \leq \langle u, \dot A^{(1)}(t) u \rangle + \int_I C |U'(y)| (|\nabla \psi(y)|^{2} + |\nabla \psi_{A}(y)|^{2})\dd y,
\end{align}
where 
\begin{align}
(-k^{2}+(\de_{y}-iktU')^{2})\psi_{A}= A u
\end{align} 
and we used that $|B|+|B'| \leq C |U'|$.
Using Lemma \ref{lem:H1elliptic}, we compare $\psi$ with our previous definition
of stream function and using Lemma \ref{lem:local_elliptic} by localized stream
functions.
However, $ A^{(1)}(t)$ is constructed in just such a way that
\begin{align}
  \int |U'| (|\nabla \psi_{j}(y)|^{2}  + |\nabla \psi_{j,A}(y)|^{2})\dd y \leq -C \langle u, \chi_{j}\dot A^{(1)}_{j}(t) \chi_{j} u \rangle.
\end{align}
Thus, under a smallness condition on $C$, and thanks to \eqref{eq:signAdot},
\begin{align}
  \ddt I \leq \eps \langle u, \dot A^{(1)}(t) u\rangle \leq 0
\end{align}
and $L^{2}$ stability of the solution operator $S(t_{2},t_{1})$ follows by
noting that $I(t)$ is a Lyapunov functional comparable to the $L^{2}$ energy.

\medskip

\noindent \underline{$H^{1}$ stability}:
We follow the same approach as in Section \ref{sec:higher} and consider the
equation satisfied by $\dg u$, that is
\begin{align}
&\de_t \dg u= ikB \dg \psi + ik (\dg  B)\psi, \\
&\dg \left(-k^{2}+(\de_{y}-iktU')^{2}-4 (\de_{y}-iktU') \frac{\dc^{(1)}}{\dc}-\left(\frac{\dc^{(1)}}{\dc}\right)^2-2\frac{\dc^{(2)}}{d}\right)\psi  =\dg u.
\end{align}
Computing the commutator in the equation satisfied by $\psi$ as in Proposition \ref{prop:DF1}, we again obtain a
modified elliptic operator and can express the Dirichlet boundary data
$\dg \psi|_{y=\a,]b}$ in terms of testing against homogeneous solutions.
The result hence follows by constructing a Lyapunov functional as in the proof
of Proposition \ref{prop:DF1}.

\medskip

\noindent \underline{Conjugation}:
For the second statement we note that the time-dependence of all coefficient
functions is given by conjugation with $\e^{ikt U}$. Hence further
conjugation with $\e^{ik \tau U}$ corresponds to a time-shift and 
\begin{align}
\e^{-ik  \tau U} S(t_{2},t_{1}) \e^{ik \tau U} u_{0} = S(t_{2}-\tau, t_{1}-\tau) u_{0}.
\end{align}
\end{proof}

Following a similar strategy as in \cite{Zill6}, we use this linear propagator
to obtain a more explicit characterization of $\nu_\bullet$.

\begin{lemma}
  \label{lem:nu_splitting}
  Let $\nu_\bullet$ be as in Lemma \ref{lem:Db_splitting} and let $S(t_{1},t_{2})$ be as
  in Lemma \ref{lem:evolution}.
  We define $\nu^{(1)}_\bullet$ by
    \begin{align}\label{eq:nu1defi}
    \nu^{(1)}_\bullet(t)
    &= ik\frac{\dc\omega^{in}}{U'}\bigg|_{y=\bullet} \int_{0}^{t} S(t,\tau)  BU'\e^{{ik\tau(U(y)-U(\bullet))}}\tilde{h}_{\bullet}\dd \tau\notag\\
    &= ik\frac{\dc\omega^{in}}{U'}\bigg|_{y=\bullet} \int_{0}^{t} \e^{{ik\tau(U(y)-U(\bullet))}}S(t-\tau,0)  BU'\tilde{h}_{\bullet}\dd \tau.
  \end{align}
  Then $\nu^{(2)}_\bullet=\nu_\bullet-\nu^{(1)}_\bullet$ satisfies
\begin{align}\label{eq:formofnu2}
\de_{t}  \nu_\bullet^{(2)} &=  ikB \psi_\bullet + \frac{B }{\dc(\a)}\l \frac{\tilde{h}^\star_\a}{\dc tU'}  \e^{ikt(U(y)-U(\a))},\nu_\a \r
 \e^{{ikt(U(y)-U(\a))}}\tilde{\tilde{h}}_{\a}\notag \\
&\quad+ \frac{B }{\dc(\b)}\l \frac{\tilde{h}^\star_\b}{\dc tU'}  \e^{ikt(U(y)-U(\b))},\nu_\b \r    \e^{{ikt(U(y)-U(\b))}}\tilde{\tilde{h}}_{\b}\notag\\
&\quad- \dc(\a) \l \nu_\bullet,\frac{h_\a}{ik (\dc tU')^2}  \e^{ikt(U(y)-U(\a))} \r   \dg (B\e^{{ikt(U(y)-U(\a))}}\tilde{h}_{\a})\notag \\
&\quad+  \dc(\b) \l \nu_\bullet,\frac{h_\b}{ik (\dc tU')^2}  \e^{ikt(U(y)-U(\b))} \r    \dg (B\e^{{ikt(U(y)-U(\b))}}\tilde{h}_{\b})
\end{align}
with initial condition $\nu_{\bullet}^{(2)} |_{t=0}=0$, where
\begin{align}
&\left(-k^{2}+(\de_{y}-iktU')^{2}-4 (\de_{y}-iktU') \frac{\dc^{(1)}}{\dc}-\left(\frac{\dc^{(1)}}{\dc}\right)^2-2\frac{\dc^{(2)}}{d}\right)\psi_\bullet= \nu_\bullet^{(2)} .
\end{align}
with boundary conditions $\psi_\bullet|_{y=\a,\b}=0$.
\end{lemma}

\begin{proof}
We note that by definition of $S(\cdot , \cdot)$ and using a Duhamel-type formula, $\nu^{(1)}_{\bullet}$ is a solution of

\begin{align}
& \de_t \nu^{(1)}_{\bullet} = ik B \psi+ ik\frac{\dc\omega^{in}}{U'}\bigg|_{y=\bullet}  BU'\e^{{ik\tau(U(y)-U(\bullet))}}\tilde{h}_{\bullet}, \\
& \left(-k^{2}+(\de_{y}-iktU')^{2}-4 (\de_{y}-iktU') \frac{\dc^{(1)}}{\dc}-\left(\frac{\dc^{(1)}}{\dc}\right)^2-2\frac{\dc^{(2)}}{d}\right)\psi  = \nu^{(1)}_{\bullet} \\
& \psi |_{y=a,b}=0
\end{align}
The result hence follows by linearity. 
\end{proof}
The following proposition states stability results for $\nu^{(1)}_{\bullet}$.

\begin{proposition}
  \label{prop:nu}
  Let $\nu^{(1)}_{\bullet}$ be as in Lemma \ref{lem:nu_splitting}.
  Then it holds that
  \begin{align}
    \|\min(\sqrt{|U(y)-U(\bullet)|},1) \nu^{(1)}_{\bullet}(t)\|_{L^{2}} &\leq C  \left|\frac{\dc\omega^{in}}{U'}\bigg|_{y=\bullet}\right| \log(1+t), \label{eq:weightbd1} \\
    \|\min(\sqrt{|U'||U(y)-U(\bullet)|},1) \nu^{(1)}_{\bullet}(t)\|_{L^{2}} &\leq C \left|\frac{\dc\omega^{in}}{U'}\bigg|_{y=\bullet}\right|,\label{eq:weightbd2} \\
    \| \nu^{(1)}_{\bullet}(t)\|_{L^{1}} &\leq C \left|\frac{\dc\omega^{in}}{U'}\bigg|_{y=\bullet}\right| \log(1+t),\label{eq:weightbd3}
  \end{align}
  and unless $\frac{\dc\omega^{in}}{U'}\big|_{y=\bullet}=0$, the unweighted $L^{p}$ norms diverge
  to infinity as $t \rightarrow \infty$.
\end{proposition}
Before coming to the proof of this proposition, we show how it further implies
stability of $\nu^{(2)}_{\bullet}$ and thus concludes our proof of Theorem \ref{thm:L2}. From \eqref{eq:formofnu2}, the most challenging
term comes from last two pieces, when $\dg$ hits the exponential and therefore deteriorates the decay. We show how handling
this term is possible if some decay is known.
\begin{lemma}
  \label{lem:suboptimal}
  Let $\nu^{(2)}_{\bullet}$ be as in Lemma \ref{lem:nu_splitting} and suppose that for
  some $\beta>0$ it holds that
  \begin{align}
    \left|\l \nu^{(1)}_\bullet,\frac{h_\a}{ t(\dc U')^2}  \e^{ikt(U(y)-U(\a))} \r \right| +
    \left|\l \nu^{(1)}_\bullet,\frac{h_\b}{ t(\dc U')^2}  \e^{ikt(U(y)-U(\b))} \r \right| 
    \leq C \langle t \rangle^{-\beta}.
  \end{align}
  Then it holds that for any $\delta>0$
  \begin{align}
    \|\nu^{(2)}_{\bullet}(t)\|_{L^{2}}^2 \leq  C_{\delta} \|\omega^{in}\|^2_{H^{2}} + C_{\delta}\left|\frac{\dc\omega^{in}}{U'}\bigg|_{y=\bullet}\right|^2 \int_{0}^{t} \langle \tau \rangle^{-2 \beta+\delta} \dd\tau. 
  \end{align}
\end{lemma}

We in particular note that $\beta>1/2$ yields a uniform bound and thus
Proposition \ref{prop:nu} yields stability of $\nu_{\bullet}^{(2)}$ as well.
If  we only assume $\beta=1/2$, this yields an upper estimate by $C \langle t \rangle^{\delta}$.

\begin{proof} 
We use the same Lyapunov functional construction as in Proposition \ref{prop:gamma}, considering the functional 
\begin{align}
I(t)=\l A_1(t)\nu^{(2)}_\bullet,\nu^{(2)}_\bullet \r+ C_{1}\langle F^{(1)}, A_{1}(t) F^{(1)} \rangle + C_{2}  \langle F(t), A_{1}(t) F(t) \rangle.
\end{align}
Referring to \eqref{eq:formofnu2}, one of the problematic terms reads
\begin{align}
&\dc(\a) \left|\l \nu^{(1)}_\bullet,\frac{h_\a}{t (\dc U')^2}  \e^{ikt(U(y)-U(\a))} \r \l A_1(t)\nu^{(2)}_\bullet ,U'  B\tilde{h}_{\a}\e^{{ikt(U(y)-U(\a))}}\r\right|\notag\\
&\qquad\leq C\langle t \rangle^{-\beta} \left| \l A_1(t)\nu^{(2)}_\bullet ,U'  B\tilde{h}_{\a}\e^{{ikt(U(y)-U(\a))}}\r\right|.
\end{align}
Using Young's inequality with one factor $\l t\r^{-\beta+\delta/2}$ and Lemma \ref{lem:coefficientsF2}, we can absorb
\begin{align}
t^{-\delta} \left| \l A_1(t)\nu^{(2)}_\bullet ,U'  B\tilde{h}_{\a}\e^{{ikt(U(y)-U(\a))}}\r\right|^{2}
\end{align}
for $t$ being sufficiently large to satisfy the smallness assumptions as in the proof of Proposition \ref{prop:DF1}.
  The result hence follows by noting that our Lyapunov functional controls $\|\nu^{(2)}_{\bullet}(t)\|_{L^{2}}^{2}$ and satisfies
  \begin{align*}
    \ddt I(t) \leq C \l t\r^{-2 \beta + \delta}, 
  \end{align*}
  and integrating in time. 
\end{proof}

\subsection{Proof of Proposition \ref{prop:nu}}
From \eqref{eq:nu1defi} and the fact that  $S(t-\tau ,0)$ is an operator from $L^{2}$ to $L^{2}$ with uniformly
bounded operator norm, it follows that
\begin{align}
  \|\nu^{(1)}_{\bullet}(t)\|_{L^{2}} \leq C t\left| \frac{\dc\omega^{in}}{U'}\bigg|_{y=\bullet}\right|.
\end{align}
In order to improve this estimate, we make use of the oscillation of $\e^{{ik\tau
    (U(y)-U(\bullet))}}$. Using \eqref{eq:nu1defi} once more and integrating by parts, we find
\begin{align}
\nu^{(1)}_\bullet(t)
&=ik\frac{\dc\omega^{in}}{U'}\bigg|_{y=\bullet} \int_{0}^{t} \e^{{ik\tau(U(y)-U(\bullet))}}S(t-\tau,0)  BU'\tilde{h}_{\bullet}\dd \tau\notag\\
&=\frac{\dc\omega^{in}}{U'}\bigg|_{y=\bullet} \int_{0}^{t} \de_\tau(\e^{{ik\tau(U(y)-U(\bullet))}}-1)\frac{1}{U(y)-U(\bullet)}S(t-\tau,0)  BU'\tilde{h}_{\bullet}\dd \tau\notag\\
&=\frac{\dc\omega^{in}}{U'}\bigg|_{y=\bullet} \frac{\e^{ikt(U(y)-U(\bullet))}-1}{U(y)-U(\bullet)} BU'\tilde{h}_{\bullet}
- \frac{\dc\omega^{in}}{U'}\bigg|_{y=\bullet} \int_{0}^{t} \frac{\e^{ik\tau(U(y)-U(\bullet))}-1}{U(y)-U(\bullet)}\de_\tau S(t-\tau,0)  BU'\tilde{h}_{\bullet}\dd \tau
\label{eq:lasttocontrol}  
\end{align}    
For the first term, the weighted  bounds \eqref{eq:weightbd1}-\eqref{eq:weightbd2} in $L^{2}$ follow by direct
computation, since
\begin{align}
   \|BU'\tilde{h}_{\bullet}\|_{L^{2}}< \infty.
\end{align}
Similarly, denoting by $B_{1}(\bullet)$ the unit ball centered at $y=\bullet$, we have
\begin{align}
  \| \frac{\e^{ikt(U(y)-U(\bullet))}-1}{U(y)-U(\bullet)} BU'\tilde{h}_{\bullet} \|_{L^{1}} \leq C \| \frac{\e^{ikt(U(y)-U(\bullet))}-1}{U(y)-U(\bullet)} BU'\tilde{h}_{\bullet}\|_{L^{1}(B_{1}(\bullet))} + C \|BU'\tilde{h}_{\bullet} \|_{L^{1}}
\end{align}
and the first term can be compared to $\|\frac{\e^{ikty}-1}{y}\|_{L^{1}} \approx
\log(1+t)$, while the second is bounded.
It hence only remains to estimate the integral term.
Using an integration by parts argument similar as in \cite{Lin-Zeng} and
\cite{Zill6}, we show the following uniform damping estimate.
\begin{lemma}
  \label{lem:damp}
  Let $U$ be mildly degenerate and $u_{0} \in H^{1}$. Let further $S(t_{1},t_{2})$
  be as in Lemma \ref{lem:evolution}. Then for any $t>0$ it holds that
  \begin{align*}
    \|\p_{t} S(t,0) u_{0}\|_{L^{2}} &\leq C  \frac{1}{t} \|S(t,0) u_{0}\|_{H^{1}} \\
    \|U' \p_{t} S(t,0) u_{0}\|_{L^{2}} &\leq C \frac{1}{t^{2}}  (\|S(t,0) u_{0}\|_{H^{1}} + \|\min(U(y)-U(\bullet), 1) \p_{y}^{2} S(t,0)u_{0} \|_{L^{2}}.
  \end{align*}
\end{lemma}

\begin{proof}
  Following the notation of the definition of $S(t,0)$, we denote $u(t)=S(t,0)u_0$
  and $\dt S(t,0)u_{0}= ik B \psi$.
  Then, by \eqref{eq:mildy2}, we obtain that
  \begin{align*}
  \|\dt S(t,0)u_{0}\|_{L^{2}}^{2}= \int_I |ik B(y) \psi(y)|^{2}\dd y = \int |B(y)|^{2} |k \psi(y)|^{2}\dd y \leq C^{2} \int |U'(y)|^{2} |\nabla_{t,k} \psi(y)|^{2}\dd y.
  \end{align*}
  Furthermore, by integration by parts and the definition of $\psi$ in \eqref{eq:psidefi}
  we proceed as in Lemma \ref{lem:H2elliptic} to obtain that
  \begin{align}
    -\int |U'|^{2} u \psi &= -\int |U'|^{2} \left(-k^{2}+(\de_{y}-iktU')^{2}-4(\de_{y}-iktU') \frac{\dc^{(1)}}{\dc}-\left(\frac{\dc^{(1)}}{\dc}\right)^2-2\frac{\dc^{(2)}}{d}\right)\psi  \psi \\
     &\geq c  \int |U'|^{2}|\nabla_{t} \psi|^{2} - \p_{y}(|U'|^{2}/2) |\psi|^{2},
  \end{align}
%
  which is positive definite since $U$ is mildly degenerate (c.f. Lemma \ref{lem:H2elliptic}).
  Hence, it suffices to estimate the left-hand side from above as 
  \begin{align*}
    -\int_I |U'|^{2} u \psi &=  \int_I |U'|^{2} u \e^{iktU} \e^{-iktU} \psi=  \int_I \frac{\de_y\e^{iktU}}{ikt} U' u  \e^{-iktU} \psi\\
    &=-  \int_I \frac{\e^{iktU}}{ikt} \de_y(U' u  \e^{-iktU} \psi)\leq  \frac{1}{kt}\|u\|_{H^{1}} \| \| U' \e^{-iktU} \psi\|_{H^{1}}. 
  \end{align*}
  Combining the lower and upper bound, we hence obtain that
  \begin{align*}
    \|\dt S(t,0)u_{0}\|_{L^{2}} \leq \sqrt{ \int_I |ik B(y) \psi(y)|^{2}\dd y} \leq \frac{C}{kt}\|u\|_{H^{1}},
  \end{align*}
  which is the first decay estimate.

  Similarly, for the higher decay estimates we introduce a potential for $ik
  \psi$ by
  \begin{align*}
    \sigma=\e^{-iktU}(-\Delta_k)^{-1}\e^{iktU}ik\psi(t,y),
  \end{align*}
  where we require that $\sigma$ has vanishing Dirichlet data.
  We then compute that 
  \begin{align*}
    \int |U'|^{4}|ik\psi|^{2} \approx \int |U'|^{4} k^{2} \sigma ik u = \int k^{2}|U'|^{4}e^{-iktU} \overline{\hat{\sigma}}e^{iktU}\hat{u} \\
    = \int\e^{iktU}  \p_{y} \frac{1}{|U'|} \p_{y} (|U'|^{3}e^{-iktU}\overline{\hat{\sigma}} \hat{u}) \\
    \leq C \||U'|^{2}\e^{-iktU}\sigma\|_{H^{2}} (\|u\|_{H^{1}} + \|\min(U(y)-U(\bullet), 1) \p_{y}^{2}u\|_{L^{2}}),
  \end{align*}
  where we used Hardy's inequality and the Dirichlet data of $\sigma$ to obtain
  the weighted $H^{2}$ norm.
  The result hence follows by noting that, due to elliptic regularity of the
  Laplacian, the weighted $H^{2}$ norm of $\sigma$ is controlled by the weighted
  $L^{2}$ norm of $\psi$.
\end{proof}

We further note that our proof only used the oscillation of $\e^{iktU}$ and the
elliptic regularity of the (modified) stream function map $u \mapsto \psi$.
Hence, with minor modifications to the constants it also holds for the usual
definition of stream function and the definition introduced in Section
\ref{sec:higher}.

\begin{corollary}
  Let $\omega^{in} \in H^{1}$ with $\int \omega^{in}(x) \dd x =0$ and suppose that $\|F(t)\|_{H^{1}}$ is uniformly
  bounded. Then for any $t>0$
  \begin{align*}
    \int |U'| |v|^{2} \leq \frac{C}{t} \|\omega^{in}\|_{H^{1}}.
  \end{align*}
  If further $\omega^{in} \in H^{2}$ and  $\|\min(|U(y)-U(a)|, |U(y)-U(b)|,
  1)\p_{y}^{2}F(t)\|_{L^{2}} \leq C \|\omega^{in}\|_{H^2}$, then also
  \begin{align*}
    \int |U'|^{2} |v|^{2} \leq \frac{C}{t^{2}} \|\omega^{in}\|_{H^{2}}.
  \end{align*}
\end{corollary}

\begin{proof}
 The same method of proof works if we consider map $\omega^{in} \mapsto F(t)$ instead
 of $S(t,0)$.
\end{proof}

In order to apply the second of the preceding results, we require control of the
evolution of $S(t,0)$ in $H^{2}$, which is formulated as the following lemma.
\begin{lemma}
  \label{lem:higherS}
  Let $S(t,0)$ be as in Lemma \ref{lem:evolution} and suppose that $u \in H^{2}$, then also
  \begin{align*}
    \|\min(U(y)-U(a), U(y)-U(b), 0) \p_{y}^{2}S(t,0) u\|_{L^{2}} \leq  C (1+|t|)^{\delta} \|u\|_{H^{2}}.
  \end{align*}
\end{lemma}

Before coming to the proof of Lemma \ref{lem:higherS}, we show how it can be
used to complete our proof of Proposition \ref{prop:nu}.

\begin{proof}[Proof of Proposition \ref{prop:nu}]
 In view of \eqref{eq:lasttocontrol} (and controlling by Gronwall's Lemma
   for small times), it suffices to control the integral term
  \begin{align}
   \int_{1}^{t} \frac{\e^{ik\tau(U(y)-U(\bullet))}-1}{U(y)-U(\bullet)}\de_\tau S(t-\tau,0)  BU'\tilde{h}_{\bullet}\dd \tau.
  \end{align}
  Using H\"older's inequality, Lemma \ref{lem:damp} and \ref{lem:higherS}, we obtain 
\begin{align}
  & \left\| \min(U',1)   \int_{1}^{t} \frac{\e^{ik\tau(U(y)-U(\bullet))}-1}{U(y)-U(\bullet)}\de_\tau S(t-\tau,0)  BU'\tilde{h}_{\bullet}\dd \tau \right\|_{L^{1}}\notag \\
  &\leq C\int_{1}^{t} \|\frac{\e^{ik\tau(U(y)-U(\bullet))}-1}{U(y)-U(\bullet)}\|_{L^{2}} \frac{1}{(t-\tau)^{2}} \Big(\|\min(U(y)-U(a), U(y)-U(b), 1) \p_{y}^{2}S(t,0) BU'\tilde{h}_{\bullet}\|_{L^{2}}\notag\\
  &\qquad+ \|S(t,0) BU'\tilde{h}_{\bullet}\|_{H^{1}}\Big)   \dd\tau\notag \\
  &\leq C \int_{1}^{t} \frac{ \tau^{\delta} \sqrt{\tau}}{\tau^{2}} d\tau \|BU'\tilde{h}_{\bullet}\|_{H^{2}} \leq C. 
\end{align}
Similarly, we conclude 
\begin{align}
  &\left\| \int_{1}^{t} \min(|U(y)-U(\bullet)|, 1) \frac{\e^{ik\tau(U(y)-U(\bullet))}-1}{U(y)-U(\bullet)} \p_{\tau} S(t-\tau,0) BU'\tilde{h}_{\bullet} \dd\tau \right\|_{L^{1}} \notag\\
  &\leq C \int_{1}^{t} \|\p_{\tau} S(\tau,0) BU'\tilde{h}_{\bullet}\|_{L^{2}} \dd\tau \notag\\
  &\leq C \|BU'\tilde{h}_{\bullet}\|_{H^{2}} \int_{1}^{t} \frac{1}{\tau^{2}} \dd\tau \leq C < \infty. 
\end{align}
The integrals from $0$ to $1$ can controlled using the rough linear growth bound
established at the beginning of this section.
\end{proof}

It hence remains to prove Lemma \ref{lem:higherS}, for which we use a bootstrap approach.

\subsection{Proof of Lemma \ref{lem:higherS}}
We show that the map $\omega^{in}\mapsto F(t)$ satisfies a similar growth
bound.
Repeating and modifying this argument slightly, the growth bound for $S(t,0)$
follows.

Recalling \eqref{eq:lasttocontrol}, we infer that 
\begin{align}
&\left|\l \nu^{(1)}_\bullet,\frac{h_\a}{ t(\dc U')^2}  \e^{ikt(U(y)-U(\a))} \r\right|\\
&\quad\leq \left|\frac{\dc\omega^{in}}{U'}\bigg|_{y=\bullet}\l  \frac{\e^{ikt(U(y)-U(\bullet))}-1}{U(y)-U(\bullet)} BU'\tilde{h}_{\bullet},\frac{h_\a}{ t(\dc U')^2}  \e^{ikt(U(y)-U(\a))} \r\right|\\
&\qquad+\left|\frac{\dc\omega^{in}}{U'}\bigg|_{y=\bullet} \int_{0}^{t} \l\frac{\e^{ik\tau(U(y)-U(\bullet))}-1}{U(y)-U(\bullet)}\de_\tau S(t-\tau,0)  BU'\tilde{h}_{\bullet}\dd \tau ,\frac{h_\a}{ t(\dc U')^2}  \e^{ikt(U(y)-U(\a))} \r\right|
\end{align}
For the first term, we may estimate
\begin{align}
  | \l \frac{\e^{ikt(U(y)-U(\bullet))}-1}{U(y)-U(\bullet)} BU'\tilde{h}_{\bullet}, \frac{h_\a}{ t(\dc U')^2}  \e^{ikt(U(y)-U(\a))}\r | \leq C \frac{1}{t} \|\frac{\e^{ikt(U(y)-U(\bullet))}-1}{U(y)-U(\bullet)}  \|_{L^{1}} \leq C\frac{\log(1+t)}{t}.
\end{align}
For the second term, we instead use Lemma \ref{lem:damp} and the estimate
\begin{align*}
  \| \frac{\e^{ikt(U(y)-U(\bullet))}-1}{U(y)-U(\bullet)} \|_{L^{2}} \leq \sqrt{\tau}, \\
   \|\de_\tau S(t-\tau,0)  BU'\tilde{h}_{\bullet}\|_{L^{2}} \leq \frac{1}{\tau} \|S(t-\tau,0)  BU'\tilde{h}_{\bullet}\|_{H^{1}} \leq \frac{C}{\tau} \| BU'\tilde{h}_{\bullet}\|_{H^{1}} \leq \frac{C}{\tau}.
\end{align*}
Combined with the $\frac{1}{iktU'}$ factor, we hence obtain a decay rate
$\frac{1}{\sqrt{t}}$ and thus by Lemma \ref{lem:suboptimal} we obtain a growth
bound
\begin{align}
  \|\nu^{(2)}_\bullet(t)\|_{L^{2}} \leq C_{\delta} \|\omega^{in}\|_{H^{2}} + C_{\delta} \left|\frac{\dc\omega^{in}}{U'}\bigg|_{y=\bullet} \right| t^{\delta}.
\end{align}
Similarly, we obtain that
\begin{align}
  \|\min(U(y)-U(\bullet), 1) \nu^{(1)}_{\bullet}(T)\|_{L^{2}} &\leq \left|\frac{\dc\omega^{in}}{U'}\bigg|_{y=\bullet}  \right|
  \left( \|BU'\tilde{h}_{\bullet}\|_{L^{2}}  + \int_1^{T}\|\p_{\tau} S(t-\tau,0) BU'\tilde{h}_{\bullet}\|_{L^{2}} \dd\tau \right) \\
  & \leq \left|\frac{\dc\omega^{in}}{U'}\bigg|_{y=\bullet}  \right| (C+\log(1+T)).
\end{align}
Combining these results we obtain that 
\begin{align}
 H^{2} \ni \omega^{in} \mapsto \min(U(y)-U(\bullet), 1) \dg^{2}F(T) \in L^{2}
\end{align}
satisfies a sub-optimal growth bound:
\begin{align}
  \|\min(U(y)-U(\bullet), 1) \dg^{2}F(T) \|_{L^{2}} \leq C \langle t \rangle^{\delta} \|\omega^{in}\|_{H^{2}}.
\end{align}
However, repeating the same argument with $S(t,0)$ in place of $\omega^{in} \mapsto
F(t)$, we also obtain that
\begin{align}
  \|\min(U(y)-U(\bullet), 1) \dg^{2} S(t,0)u_{0}\|_{L^{2}} \leq C \langle t \rangle^{\delta} \|u_{0}\|_{H^{2}}.
\end{align}
Thus, me revisit the above decay estimates and use that
\begin{align*}
    \| U' \p_{\tau} S(\tau,0) BU'\tilde{h}_{\bullet}\|_{L^{2}} 
  &\leq \frac{1}{\tau^{2}} (\|S(\tau,0) BU'\tilde{h}_{\bullet}\|_{H^{1}} + \|\min(U(y)-U(\bullet), 1) \dg^{2} S(\tau,0) BU'\tilde{h}_{\bullet}\|_{H^{1}}) \\
  &\leq C \frac{\langle \tau \rangle^{\delta}}{\tau^{2}}.
\end{align*}
This yields integrability in time and therefore 
\begin{align}
\left|\l \nu^{(1)}_\bullet,\frac{h_\a}{ t(\dc U')^2}  \e^{ikt(U(y)-U(\a))} \r\right|\leq C\frac{\log(1+t)}{t} \leq Ct^{-1+\delta}
\end{align}
and by Lemma \ref{lem:suboptimal}
\begin{align}
    \|\nu^{(2)}_{\bullet}(t)\|_{L^{2}}^2 \leq  C_{\delta} \|\omega^{in}\|^2_{H^{2}} + C\left|\frac{\dc\omega^{in}}{U'}\bigg|_{y=\bullet}\right|^2.
\end{align}
Similarly, the control of $\nu^{(1)}_{\bullet}$ improves to 
\begin{align}
  \|\min(U'(U(y)-U(\bullet)), 1) \nu^{(1)}_{\bullet}(t)\|_{L^{2}}|_{t=1}^{T} \leq \int_1^{T}\|U' \p_{\tau} S(\tau,0)   BU'\tilde{h}_{\bullet}\|_{L^{2}} d\tau \\
\leq \int_1^{T} \frac{1}{\langle \tau \rangle^{2}} \tau^{\delta} \| BU'\tilde{h}_{\bullet}\|_{H^{2}} \dd\tau \leq C.
\end{align}
This concludes the proof of Lemma \ref{lem:higherS} and thus of Proposition \ref{prop:nu}.

\section{Application to circular flows}\label{sec:circular}

One of the main examples of interest in the theory of inviscid damping for
circular flows is given by the question of stability of a point vortex
$\omega(t)= \delta_{x_{0}}$.
On the level of the velocity, this flow corresponds to the mildly degenerate
angular velocity $r^{-1} e_{\theta}$, which is a Taylor-Couette flow.
Hence, at the linearized level, the equations reduce to a transport problem and
are explicitly solvable.
However, similar to Couette flow, introducing any perturbation to this profile
one loses this explicit solution. In particular, any nonlinear stability result
would have to account for the fact that a perturbation might cause the point
vortex to move and introduce a perturbation to the (moving) circular flow around
the point vortex.
The presently considered setting, in a sense, fixes the point vortex at the
origin and studies the linear stability problem for perturbations.
In \cite{Zill6}, the first author has established stability in the
case of a compact domain $\T \times (R_{1},R_{2})$ with $0<R_{1}<R_{2}<\infty$.
The result for the unbounded domain case established in this paper introduces
several additional challenges:
\begin{itemize}
\item The physically natural spaces in the circular setting are given by
  $L^{2}(r\dd r \dd\theta)$-based spaces. The weight $r$ has hence to be taken into
  account in the construction of the Lyapunov functional and further degenerates
  as $r \downarrow 0$ and $r \uparrow \infty$. Thus, even in the case of a
  bilipschitz flow profile, weighted spaces require the use of our localized
  estimates. 
\item The Biot-Savart law in polar coordinates also includes a degenerate
  dependence on $r$, which then further has to be studied with degenerate
  weights. 
\item While in the case of the velocity corresponding exactly to a point vortex,
  $B=0$ and the dynamics are trivial, small perturbations require one to deal
  with both $U$ and $B$ being degenerate. 
\end{itemize}
Using our localization methods (to dyadic annuli in Cartesian coordinates), our
methods allow us to study mildly degenerate circular flows, which may asymptotically behave like power laws.
That is, we for instance consider profiles $U(r)\sim r^{\alpha}$ for $r \downarrow 0$ and
$U(r) \sim r^\beta$ for $r \uparrow \infty$, where $\alpha$ and $\beta$ may
differ.

The results of the previous sections, provide stability of (c.f. Section \ref{sec:linEuler})
\begin{align}
  \omega(s, \theta, t)= \e^{2s}\omega(r=\e^{s},\theta, t)
\end{align}
in the unweighted $L^{2}(\dd s \dd\theta)$ based Sobolev spaces.
However, from a physical point of view it would be more natural to control
\begin{align}
  \iint |\omega|^{2} r \dd r \dd  \theta = \iint \e^{2s} \omega(r=\e^{s}, \theta) \dd s \dd \theta &= \int |\omega|^{2} \e^{-2s} \dd s \dd \theta, \\
  \int |v|^2 \dd x \dd y = - \int \psi \omega r \dd r \dd \theta = -\int \psi \omega \e^{2s} \dd s \dd \theta = \int \psi \omega \dd s \dd \theta &= \int |\nabla\psi|^2 \dd s \dd \theta
\end{align}
Instead of considering weighted spaces, from a technical perspective we prefer
to study  $\omega_\star= \e^{-s}\omega$, $\psi_\star = \e^{-s} \psi$ in unweighted spaces.
Again denoting $s, \theta$ by $y,x$, our equations are then given by
  \begin{align}
    \dt \omega_\star + U(y)\p_x \omega_{\star} &= B\p_x \psi_{\star}, \\
    (\p_{x}^2+ (\p_y+1)^2) \psi_{\star} &= \omega_{\star}.
  \end{align}
  Thus the only modification of our equations occurs in the elliptic operator.
  We remark that by the triangle inequality
  \begin{align}
   - \int \psi_{\star} (\p_{x}^2+ (\p_y+1)^2) \psi_{\star} = \int |\p_{x}\psi_{\star}|^{2} + |(\p_{y}+1)\psi_{\star}|^{2} \geq \int |\p_{x}\psi_{\star}|^{2} - |\psi_{\star}|^{2} + |\p_{y}\psi_{\star}|^{2}.
  \end{align}
  Hence, this bilinear form is positive definite if we have a sufficient
  spectral gap in $x$ to absorb $-|\phi_{\star}|^{2}$.
  Using this triangle inequality approach, we similarly note that commutators
  for higher derivatives like
  \begin{align*}
  [\dg ,(\p_{x}^2+ (\p_y+ikt U')^2 + 2(\p_y+ikt U') + 1 )]
  \end{align*}
  are mostly unchanged by this modification and homogeneous solutions differ by
  multiplication by $\e^{s}$.
  Thus, with minor modifications, we obtain the following stability result.
  \begin{theorem}
   Under the same assumptions as in Theorem \ref{thm:L2} with a possible small increase
   of $k$ in condition \ref{H1}, stability also holds in the
   physically natural  $L^{2}(\e^{-2s}\dd s \dd \theta)$ based spaces.  
  \end{theorem}

\section*{Acknowledgments}
M. Coti Zelati was partially supported by NSF grant DMS-1713886.
C. Zillinger has been supported by a travel grant of the Simon's foundation.


\end{document}